 \documentclass [12pt]{amsart}

 \usepackage{amsmath, accents} %

\usepackage{amssymb,amsxtra,amsfonts}
\usepackage{epsf}              %
\usepackage{epsfig}             %
\usepackage{graphics}
\usepackage{epigraph}

 \usepackage{tikz}
 \usepackage{wrapfig}

\usepackage[colorlinks]{hyperref}
\usepackage{combelow}   %

\usepackage{scalerel}   %

 \usepackage{etoolbox}

\usepackage{xspace} 

\newenvironment{ppb}[1]
{\ \!\!\!\!\!\!\!\!\!\!\!\!\!\!\!\!\!\!\!\!\!\!\!\!\!\!\!\!\!\!\!\!\!\!\!\!\!\!\!\! {\bf PPB-----------------------------------------------------------------------------PPB}\newline \tiny {#1}
\  \newline\normalsize\phantom{f}\!\!\!\!\!\!\!\!\!\!\!\!\!\!\!\!\!\!\!\!\!\!\!\!\!\!\!\!\!\!\!\!\!\!\!\!\!\!\!\! {\bf 
PPB-----------------------------------------------------------------------------PPB}\newline}{}

\long\def\pb #1*/{}

\def\reE@DeclareMathSymbol#1#2#3#4{%
    \let#1=\undefined
    \DeclareMathSymbol{#1}{#2}{#3}{#4}}
\DeclareSymbolFont{symbolsC}{U}{txsyc}{m}{n}
\SetSymbolFont{symbolsC}{bold}{U}{txsyc}{bx}{n}
\DeclareFontSubstitution{U}{txsyc}{m}{n}
\reE@DeclareMathSymbol{\strictiff}{\mathrel}{symbolsC}{76}

\newcommand\beq{\begin{equation}}
\newcommand\eeq{\end{equation}}
\newcommand\bal{\begin{align*}}
\newcommand\eal{\end{align*}}   
\newcommand\bmx{\left(\begin{matrix}}
\newcommand\emx{\end{matrix}\right)}
\newcommand\bsmx{\left(\begin{smallmatrix}}
\newcommand\esmx{\end{smallmatrix}\right)}
\newcommand\bmxnp{\begin{matrix}}
\newcommand\emxnp{\end{matrix}}
\newcommand\bsmxnp{\begin{smallmatrix}}
\newcommand\esmxnp{\end{smallmatrix}}

\DeclareMathSymbol{\widehatsym}{\mathord}{largesymbols}{"62}
\newcommand\lowerwidehatsym{%
  \text{\smash{\raisebox{-1.3ex}{%
    $\widehatsym$}}}}
\newcommand\fixwidehat[1]{%
  \mathchoice
    {\accentset{\displaystyle\lowerwidehatsym}{#1}}
    {\accentset{\textstyle\lowerwidehatsym}{#1}}
    {\accentset{\scriptstyle\lowerwidehatsym}{#1}}
    {\accentset{\scriptscriptstyle\lowerwidehatsym}{#1}}
}

\newcommand{\wh}{\fixwidehat}

\renewcommand{\d}{\partial}

\providecommand{\cir}{} 
\renewcommand{\cir}[1]{\< #1 \>}   

\newcommand{\bSi}{{\bf \Si}}

\newcommand{\col}{{\mkern2mu:\mkern2mu}}
\newcommand{\onto}{\twoheadrightarrow}

\newcommand{\into}{\hookrightarrow}

\newcommand{\st}{\ \bigl\vert\ }
\providecommand{\abs}[1]{\lvert#1\rvert}

\providecommand{\levels}{\text{\rm Levels}}
\providecommand{\Levels}{\text{\rm Levels}}
\providecommand{\lcm}{\text{\rm lcm}}

\providecommand{\from}{\leftarrow}

\providecommand{\floor}[1]{\lfloor#1\rfloor}

\providecommand{\<}{\langle}
\renewcommand{\>}{\rangle}

\def\part#1{\frac{\partial\phantom{q}}{\partial#1}}

\newcommand {\flb}{\lbrack\!\lbrack}
\newcommand {\frb}{\rbrack\!\rbrack}
\newcommand {\flp}{(\!(}
\newcommand {\frp}{)\!)}

\DeclareFontFamily{U}{wncy}{}
\DeclareFontShape{U}{wncy}{m}{n}{<->wncyr10}{}
\DeclareSymbolFont{mcy}{U}{wncy}{m}{n}
\DeclareMathSymbol{\Sh}{\mathord}{mcy}{"58} 

\newcommand{\MB}{\mathcal{M}_{\text{\rm B}}}

\newcommand{\uMB}{\underline{\mathcal{M}}_{\text{\rm B}}}

\renewcommand{\MR}{\mathop{\rm MR}}

\newcommand{\Sect}{\text{\rm Sect}}

\newcommand{\Sym}{{\rm Sym}}

\newcommand{\ram}{\mathop{\rm Ram}} 
\newcommand{\Ram}{\mathop{\rm Ram}} 
\newcommand{\Irr}{\text{\rm Irr}}
\newcommand{\slope}{\mathop{\rm slope}}
\DeclareMathOperator{\pr}{pr}

\newcommand{\Map}{{\mathop{\rm Map}}}

\newcommand{\Tr}{{\mathop{\rm Tr}}}

\DeclareMathOperator{\Hom}{Hom}         
\DeclareMathOperator{\Aut}{\mathop{\rm Aut}}

\newcommand{\SL}{{\mathop{\rm SL}}}

\newcommand{\GL}{{\mathop{\rm GL}}}

\DeclareMathOperator{\End}{End}

\DeclareMathOperator{\Katz}{Katz}
\newcommand{\diag}{{\mathop{\rm diag}}}

\newcommand{\ba}{{\bf a}}

\newcommand{\bd}{{\bf d}}

\DeclareSymbolFont{bbold}{U}{bbold}{m}{n}
\DeclareSymbolFontAlphabet{\mathbbold}{bbold}

\newcommand{\IA}{\mathbb{A}}
\newcommand{\IB}{\mathbb{B}}
\newcommand{\IC}{\mathbb{C}}

\newcommand{\II}{\mathbb{I}}

\newcommand{\IL}{\mathbb{L}}

\newcommand{\IN}{\mathbb{N}}

\newcommand{\IP}{\mathbb{P}}                                     
\newcommand{\IQ}{\mathbb{Q}}                           
\newcommand{\IR}{\mathbb{R}}

\newcommand{\IV}{\mathbb{V}}

\newcommand{\IY}{\mathbb{Y}}
\newcommand{\IZ}{\mathbb{Z}}

\newcommand{\cB}{\mathcal{B}}

\newcommand{\cF}{\mathcal{F}}

\newcommand{\cI}{\mathcal{I}}

\providecommand{\cL}{}
\renewcommand{\cL}{\mathcal{L}}

\newcommand{\cS}{\mathcal{S}}

\newcommand{\cT}{\mathcal{T}}

\newcommand{\al}{\alpha}

\newcommand{\be}{\beta}
\newcommand{\ga}{\gamma}

\newcommand {\eps}{\varepsilon}

\newcommand{\Ga}{\Gamma}

\newcommand{\La}{\Lambda}

\newcommand{\si}{\sigma}

\newcommand{\Si}{\Sigma}
\newcommand{\Th}{\Theta}

\newcommand{\ze}{\zeta}

\makeatletter
 \newlength{\typesize}
\makeatother

\newlength{\vvoff}
\newlength{\hhoff}

\newcommand{\pf}{\begin{bpf}}

\newcommand{\pfms}{\begin{bpfms}}
\newcommand{\epf}{\end{bpf}\hfill$\square$\\}           
\newcommand{\epfms}{\end{bpfms}\hfill$\square$\\}       

\newcommand{\idea}{\begin{bidea}}

\newcommand{\eidea}{\end{bidea}\hfill$\square$\\}           

\newcommand{\sk}{\begin{bsk}}    

\newcommand{\esk}{\end{bsk}\hfill$\square$\\}           
\newcommand{\sketch}{\begin{bsketch}}

\newcommand{\esketch}{\end{bsketch}\hfill$\square$\\}

\theoremstyle{plain}  
\newtheorem{hypo}{Hypothesis}[section] 
\newtheorem{thm}[hypo]{Theorem}
\newtheorem{prop}[hypo]{Proposition}

\newtheorem{cor}[hypo]{Corollary}
\newtheorem{lem}[hypo]{Lemma}

\newtheorem {defn}[hypo]{Definition}

\newtheorem {eg}[hypo]{Example}
\theoremstyle{remark}
\theoremstyle{remark}\newtheorem{remark}[hypo]{Remark}

\DeclareMathAlphabet{\mathbbmsl}{U}{bbm}{m}{sl}

\numberwithin{equation}{section}

\providecommand{\abs}[1]{\lvert#1\rvert}

\DeclareMathOperator{\Rank}{Rank}  %
 \DeclareMathOperator{\Inc}{Inc} %
 \DeclareMathOperator{\Ch}{Ch}  %
 \DeclareMathOperator{\succr}{succ}  %
 \DeclareMathOperator{\precr}{prec}  %
 \newcommand{\N}{\mathbb{N}}
 \newcommand{\Z}{\mathbb{Z}}
 \newcommand{\Q}{\mathbb{Q}}
 \newcommand{\C}{\mathbb{C}}  %

\DeclareFontFamily{U}{mathx}{\hyphenchar\font45}
\DeclareFontShape{U}{mathx}{m}{n}{<-> mathx10}{}
\DeclareSymbolFont{mathx}{U}{mathx}{m}{n}
\DeclareMathAccent{\wb}{0}{mathx}{"73}  %

\newcommand{\nonpure}{full} %
\newcommand{\Nonpure}{Full} %
\newcommand{\authorised}{inconsequential} %
\newcommand{\mit}{pointed irregular type} %
\newcommand{\mits}{pointed irregular types} %

\newcommand{\pureconfig}{\mathbf{B}}

\newcommand{\fullconfig}{\wb {\mathbf{B}}}
\newcommand{\purelwmcg}{\Ga}
\newcommand{\fulllwmcg}{\wb \Ga}
 \newcommand{\efs}{exponential factors}

\newcommand{\Ical}{\mathcal{I}}

\begin{document}

\title[Twisted local wild mapping class groups]{Twisted local wild mapping class groups:\\ configuration spaces, fission trees\\ and complex braids}

\author{Philip Boalch, Jean Dou\c{c}ot and  Gabriele Rembado}

\thanks{J.~D. is funded by FCiências.ID. G.~R. was supported by the Deutsche Forschungsgemeinschaft (DFG, German Research Foundation) under Germany's Excellence Strategy - GZ 2047/1, Projekt-ID 390685813, and is now funded by the EU under grant n. 101108575 (HORIZON-MSCA project~\href{https://ec.europa.eu/info/funding-tenders/opportunities/portal/screen/how-to-participate/org-details/999999999/project/101108575/program/43108390/details}{QuantMod})}

\begin{abstract}
Following the completion of the algebraic construction of the Poisson wild character varieties (B.--Yamakawa 2015)
one can consider their natural deformations, generalising both the mapping class group  actions on
the usual (tame) character varieties, and the G-braid groups already known to occur in the wild/irregular setting.
Here we study these {\em wild mapping class groups} in the case of arbitrary formal structure in type A.   As we will recall, this story is most naturally phrased in terms of admissible deformations of 
wild Riemann surfaces. 
The main results are: 1) the construction of configuration spaces containing all possible local deformations, 2) the definition of a combinatorial object, the ``fission forest'', of any wild Riemann surface and a proof that it gives a sharp parameterisation of all the admissible deformation classes.   As an application of 1), by considering basic examples, we show that the braid groups of all the complex reflection groups known as the generalised symmetric groups appear as wild mapping class groups.   As an application of 2), we compute the dimensions of all the (global) moduli spaces of type A wild Riemann surfaces (in fixed admissible deformation classes), a generalisation of the famous ``Riemann's count'' of the dimensions of the moduli spaces of compact Riemann surfaces.
\end{abstract}

\maketitle

 %


%

\section{Introduction}

The appearance of braid groups in 2d gauge theory is intimately related to the theory of isomonodromic deformations of linear connections and in turn to the Riemann--Hilbert problem. 
The simplest picture is to fix a Lie group $G$ such as $\GL_n(\IC)$ and a Riemann surface $\Si$ and to consider the arrow
$$\Si \quad \longmapsto \quad \MB := \Hom(\pi_1(\Si),G)/G$$
attaching the $G$-character variety (or Betti moduli space) $\MB$ to the surface.
This works well in families and one obtains lots of flat nonlinear connections in this way: 
If $\underline{\Si}\to \IB$ is a family of smooth Riemann surfaces over a base $\IB$ then the character varieties of the
fibres assemble into a fibre bundle
$$\uMB \quad \longrightarrow \quad \IB$$
over the same base $\IB$, and moreover this bundle  comes equipped with a natural complete 


\tableofcontents

\noindent
flat Ehresmann connection (or in other terms it is a ``local system of varieties'').
Integrating this nonlinear connection around loops in $\IB$ yields an action of the fundamental group
of the base on any fibre $\MB(b)$, i.e. for any basepoint $b\in \IB$ there is a homomorphism
\beq\label{eq: monod action}
\pi_1(\IB,b) \longrightarrow \Aut_{\text{Poisson}}(\MB(b)).
\eeq
For example taking $\IB$ to be the configuration space of $m$-tuples of points of $\IC$ leads to the usual $m$-string braid group action on the genus zero tame character varieties (Hurwitz action), or taking $\IB$ to be the Riemann moduli space of curves yields the  natural action of the mapping class group on the character varieties of compact Riemann surfaces.

This story has a vast generalisation involving new braidings, obtained by considering the monodromy data of irregular singular meromorphic connections, and their isomonodromic (monodromy preserving) deformations. To see this generalisation, 
suppose  $\Si$ above is actually a smooth complex algebraic curve. Then, under the Riemann--Hilbert correspondence, $\MB$ parameterises {\em regular singular} 
(or {\em tame}) algebraic connections on vector bundles on $\Si$.
The simplest nontrivial 
example is to take $\Si$ to be a four-punctured 
sphere $\IP^1\setminus\{0,t,1,\infty\}, 
\IB=\mathfrak{M}_{0,4}\cong\IC\setminus \{0,1\}$
the moduli space of ordered four-tuples of points, and $G=\SL_2(\IC)$.
Then the character varieties have 
complex dimension $6$ and 
are foliated by symplectic  leaves 
$\MB(\mathcal{C})\subset \MB$  
of complex dimension two (fixing the four local monodromy conjugacy classes), all preserved by the nonlinear connection, so we can restrict attention to these subbundles.
On the other side of the Riemann--Hilbert correspondence these nonlinear connections can be written explicitly, whence they become the second order nonlinear differential equations known as the Painlev\'e VI equations.
They control the isomonodromic deformations of rank two Fuchsian systems with four poles on $\IP^1$.
The Painlev\'e VI equations are the simplest examples of {\em nonlinear} geometric differential equations (and this story is essentially the way they were originally discovered by R. Fuchs, building on ideas of L. Fuchs and B. Riemann).
The flatness of the bundle $\uMB\to \IB$ gives the definition of ``isomonodromic'', as the underlying punctured surface varies.

The generalisation comes about by considering {\em irregular singular} (or {\em wild}) algebraic linear  connections on vector bundles on $\Si$, their
topological description furnished by the 
irregular Riemann--Hilbert correspondence
(Riemann--Hilbert--Birkhoff), involving monodromy and Stokes data, and the resulting generalisation of the Betti spaces $\MB$, the 
{\em wild character varieties}.
The initial motivation was simply to obtain the first $5$ Painlev\'e equations as integrability conditions
for linear connections. This was done by Garnier \cite{garn1912} (rewritten in a more convenient form by Jimbo--Miwa \cite{JM81}), and then generalised to 
generic irregular connections of arbitrary rank by Jimbo--Miwa--Ueno \cite{JMU81} (see their article, and those of Garnier \cite{garn1912, garn1926}, also for more detailed historical background).
The paper \cite{smid} then rewrote part of \cite{JMU81} in a more moduli theoretic language and proved all the Jimbo--Miwa--Ueno isomonodromy equations were symplectic (this involved generalising the Narasimhan, Atiyah--Bott, Goldman symplectic form to the irregular case), showing the generic wild character varieties formed a new class of holomorphic symplectic manifolds.

The key feature in the irregular case is that there are new parameters in the connections whose deformations behave exactly like the deformations of the underlying surface $\Si$. 
In brief (in the generic setting of \cite{JMU81}) one looks at connections locally isomorphic to connections of the form 
$$\nabla = d-A,\qquad A=dQ+\La\frac{dz}{z} + \text{(holomorphic)}$$ 
where $\La$ is a constant matrix and 
$Q$ (the {\em irregular type}) 
is a diagonal matrix of polynomials in $1/z$ (where $z$ is a local coordinate vanishing at the pole):
$$Q=\diag(q_1,\ldots,q_n),\qquad q_i\in x\IC[x], \quad x=1/z.$$
The generic case of \cite{JMU81} is when all the differences $q_i-q_j$ are polynomials of the same degree, and then the corresponding Stokes data is quite simple, as 
explained in \cite{BJL79, JMU81}. 
In effect the story in \cite{JMU81} then says that, in this setting, 
the base space $\IB$ of the deformations should be enriched by adding the irregular types at each pole.
The genericity condition should be preserved, and so the leading coefficient at each irregular pole should be diagonalisable with distinct eigenvalues.
If the structure group is $G=\GL_n(\IC)$ then this adds a factor of 
$$\mathfrak{t}_\text{reg} = \{v\in\IC^n\st v_i\neq v_j
\text{ if } i\neq j\}$$
to $\IB$ for each irregular pole.
This gives new braidings since the 
fundamental group of $\mathfrak{t}_\text{reg}$ is the (pure) $n$-string braid group.

The simplest global irregular example is to consider connections of the form
$$d-A,\qquad A= \left(\frac{A_0}{z^2} + \frac{B}{z}\right)dz$$
where $A_0$ is diagonal and $B$ is arbitrary.
Then the new deformation parameters are the eigenvalues of $A_0$, appearing in the irregular type $Q=-A_0/z$ at zero.
This example is especially alluring since such connections arise \cite{BJL81}  by considering the Fourier--Laplace transform of Fuchsian systems
$$\frac{d}{dz}-C,\qquad C= \sum_1^m \frac{A_i}{z-a_i}$$
whence it becomes clear that the 
eigenvalues of $A_0$ correspond exactly to the positions $\{a_i\}$ of the poles of the Fuchsian system.
This example makes it really clear that we should be thinking of the irregular type
on an equal footing to the pole positions.
Of course, most irregular connections will not be related to a regular singular connection by any such integral transform. 
(The exact statement of \cite{BJL81} was clarified in \cite{k2p} Appx. A, \cite{nlsl} Diagram 1---see also \cite{malg-book} XII, \cite{cmqv} \S2.)

The two main directions of generalisation of this story were then: 1) to replace the structure group $G$ by an arbitrary complex reductive group $G$ and thereby see that all the $G$-braid groups occur in 2d gauge theory (\cite{bafi}), and 2) to consider all the non-generic connections and their isomonodromic deformations (basic examples of this occur in the simply-laced story 
\cite{slims}, motivated by the increase in symmetry that occurs by allowing non-generic connections).

For example, in the sequence of works 
\cite{saqh, fission, gbs, twcv} 
(in increasing generality), 
the 
wild character variety $\MB=\Hom_\mathbb{S}(\Pi,G)/{\bf H}$ of {\em any} 
algebraic connection on a principal $G$-bundle on $\Si$ was constructed, as a
(finite dimensional)  algebraic Poisson variety, for any complex reductive group $G$.
The simpler $\GL_n(\IC)$ case most relevant here is reviewed in \cite{tops}. 
(This algebraic construction is complementary to the analytic proof 
\cite{wnabh} that the symplectic leaves of these Poisson varieties are hyperk\"ahler manifolds in type $A$, upgrading the complex symplectic quotient in \cite{smid} to a hyperk\"ahler quotient.)

As part of this story, the extra deformation parameters were isolated in a coordinate-free way, leading to the general definition of {\em wild Riemann surface} \cite{gbs, p12, twcv}.
The key point is that the wild character varieties form a local system of Poisson varieties
over any admissible deformation of a wild Riemann surface (so we get lots of new nonlinear flat connections generalising 
$\uMB\to \IB$ above, and encompassing  
all the Painlev\'e and Jimbo--Miwa--Ueno examples). 
This viewpoint vivifies the observation \cite{smid} that the irregular isomonodromy connections generalise the nonabelian Gauss--Manin connections. 
The admissibility condition generalises the notion of the connections remaining generic in the generic setting.

The aim of the present work is to study the admissible deformations 
of an arbitrary  wild Riemann surface in type $A$.
In other words we fix the structure group to be $G=\GL_n(\IC)$ or $\SL_n(\IC)$ and allow arbitrary formal structure at each pole.
In this setting a wild Riemann surface $\bSi$
is a triple $(\Si,\ba,\Th)$
where: 

\ 

$\bullet$\ $\Si$ is a compact Riemann surface

$\bullet$\ $\ba\subset \Si$ is a finite subset

$\bullet$\ $\Th$ is the data of a rank $n$ irregular class $\Th_a$ at each point $a\in \ba$.

\ 

\noindent 
In turn an irregular class is defined as follows (cf. \cite{twcv} Prop. 8).
Each point $a\in \Si$ canonically determines the set $\cS_a$ of Stokes circles at $a$, and 
an irregular class at $a$ is a finite multiset 
of Stokes circles, written symbolically as a finite sum:
$$\Th_a = \sum n_i I_i$$
where $n_i\ge 1$ are integer multiplicities and 
each $I_i\in \cS_a$ is a Stokes circle at $a$.
If we choose a local coordinate $z=1/x$ vanishing at $a$,
then each Stokes circle $I\in \cS_a$ can be written as 
\beq\label{eq: stcircle} 
I=\< q\>,\qquad q=\sum a_i x^{k_i}
\eeq
where the sum is finite, $a_i\in \IC$ and each $k_i$ is a rational number $>0$.
In brief the Stokes circle $\<q\>$ 
is the germ of the Riemann surface where the function $q$ becomes single valued, equipped with the germ of the function $q$ (a coordinate free definition is in \cite{twcv} Rmk. 3, generalising that in \cite{gbs}).
The {\em rank} of $\Th_a=\sum n_iI_i$ is 
$\sum n_i\ram(I_i)$ where $\ram(I)$ is the ramification number of the Stokes circle 
$I=\<q\>$ (the lowest common multiple of the denominators of the $k_i$ present in $q$).
If $\Ram(I)>1$ 
for any Stokes circle in the irregular class, 
we say that the irregular class is {\em twisted} 
(or ``ramified''), 
otherwise $\Theta_a$ is {\em untwisted}. 
For example on p.116 of Stokes' 1857 paper \cite{stokes1857} on the Airy functions,  he writes down a basis
$$u = Cx^{-1/4}\exp(-2x^{3/2})\wh F(-x^{-3/2})
\ +\  Dx^{-1/4}\exp(2x^{3/2})\wh F(x^{-3/2})$$
of formal solutions 
to the Airy equation $y''=9xy$ at $x=\infty$, 
where $\wh F$ is a formal power series. 
 In this case the Stokes circle is 
 $\<2x^{3/2}\>$, which has rank $2$ and ramification $2$, so is  twisted, and (still on p.116)   
 Stokes drew a projection of the Stokes circle to the $x$ plane, to illustrate the 
 change in dominance of the two branches of the 
 exponential factor $\exp(2x^{3/2})$. 
 This trefoil-like drawing was reproduced on the title page of \cite{twcv}, and many more such pictures can easily be drawn \cite{ssd}.
 In the generic setting one just has $\Th=\sum_1^n\<q_i\>$ if $Q=\diag(q_i)$ as above.
The original tame case arises by taking  
each Stokes circle to be the tame circle 
$\<0\>\in \cS_a$, so the irregular class is simply $n\<0\>$.
Thus the main difference in the wild case is that we have the choice of an arbitrary finite multiset of Stokes circles at each marked point
(and these give the possible essentially singular 
behaviours $\exp(q)$ of the solutions of the corresponding linear connections). 
Such functions $q$ %
with arbitrary ramification appeared in
Fabry's 1885 thesis \cite{fabry-these} p.85.


The main questions motivating this paper (and the corresponding results) 
are as follows:

\noindent{\bf Qn.1)}\   Suppose we have two rank $n$ wild Riemann surfaces $\bSi, \bSi'$. How do we decide if they are admissible deformations of each other?

In the tame case the answer is well-known: it happens if and only if the genus $g$ of $\Si$ and the number $m=\#\ba$ of marked points match up.
In the wild case we will give a complete answer to this question by defining an appropriate combinatorial object, the fission tree $\cT(\Th)$ of any 
irregular class $\Th$, so that a wild Riemann surface determines a fission forest $\mathbf{F}$  (i.e. a collection of isomorphism classes of fission trees).

\begin{thm}
Two irregular classes  $\Th,\Th'$ at a point $a\in \Si$ are admissible deformations of each other if and only if their fission trees are isomorphic: $\cT(\Th)\cong \cT(\Th')$.
Consequently two rank $n$ wild Riemann surfaces $\bSi, \bSi'$
are admissible deformations of each other if and only if the corresponding pairs $(g,\mathbf{F})$ are equal, where $\mathbf{F}$ is the
fission forest and $g$ is the genus.
\end{thm}

Thus, in essence, in the wild case the number $m$ of marked points gets upgraded to the fission forest $\mathbf{F}$ (with $m$ trees in it).
This result will be deduced from a more precise statement about irregular types (see Cor. \ref{cor: adm defm types condns}).
The pair $(g,\mathbf{F})$ is called the 
{\em topological skeleton} of the wild Riemann surface $\bSi$.
In particular this yields a simple criterion to see if two wild character varieties are isomorphic:

\begin{cor}
Suppose that $\bSi, \bSi'$ are two rank $n$ wild Riemann surfaces and
 let $\MB(\bSi)$, $\MB(\bSi')$
be the corresponding  Poisson wild character varieties 
(as in \cite{twcv, tops}, generalising \cite{gbs, saqh}).
If $\bSi, \bSi'$ have the same genus and fission forests, 
then there is an algebraic Poisson isomorphism:
$$\MB(\bSi)\ \cong \  \MB(\bSi')$$
between the corresponding Poisson wild character varieties.
In particular their hyperk\"ahler symplectic leaves are thus deformation equivalent.
\end{cor}
\pf
Since $\bSi, \bSi'$ are admissible deformations of each other, 
there is a local system of Poisson varieties over some base $\IB$ 
having $\MB(\bSi)$ and  
$\MB(\bSi')$ as two of its fibres.
This statement is proved in \cite{gbs} Thm. 10.2 in the untwisted case (see also \cite{smid} Prop. 3.8, Cor. 3.9 in the generic setting) and that proof works verbatim in the general setting of \cite{twcv}. 
Thus any path $\ga$ in $\IB$ between the corresponding two points of $\IB$ 
lifts to an algebraic Poisson 
isomorphism $\MB(\bSi)\cong\MB(\bSi')$.
\epf

We will also give a sharp characterisation of the possible fission trees 
(Defn. \ref{defn: fissiontree}, Cor. \ref{cor: realisability}), which thus implies a bound 
 on the possible 
isomorphism classes of wild character varieties (\S\ref{ssn: skeleta}).

\

\noindent{\bf Qn.2)}\  What do the admissible deformations of an arbitrary irregular class look like? What types of generalised braid groups appear in general? Are there explicit configuration spaces analogous to the simple configuration spaces
$\IC^n\setminus (\text{diagonals})$ that appear both in the tame case 
and in the generic irregular case?

To answer these, we first define the notion of pointed irregular types 
(by adding some ordering data to an irregular class, see Defn. \ref{defn: pit}) and then
give an explicit construction  of a 
configuration space $\pureconfig(Q)$ of all admissible deformations 
of a pointed irregular type  $Q$ 
(with bounded slope, i.e. bounded 
Poincar\'e--Katz rank,
 $\Katz(Q)$). 

The configuration space $\pureconfig(Q)$ is completely explicit 
and involves marking free coefficients on the fission tree, 
subject to three types of conditions 
(see \eqref{eq: explicit config} and Thm. \ref{thm: charn of realzns}). 
This  involves a truncation $\cT^\flat$ 
of the fission tree $\cT$ just above the maximal slope 
(see \S\ref{ssn: truncatedtrees}).
This helps understand the admissible deformations since there are 
natural factorisations of the configuration spaces into simple pieces:

\begin{thm}
The configuration space $\pureconfig(Q)$ of admissible deformations of any pointed irregular type $Q$ is homeomorphic to the following product over the vertices 
$\IV^\flat$ of its 
truncated fission tree $\mathcal{T^\flat}$:
\[
\pureconfig(Q)\cong\prod_{v\in \IV^\flat} \pureconfig_v(\cT)
\]
where $\pureconfig_v(\cT)$ is a point if the vertex $v$ has 
no non-empty children, and if $v$ has $n$ non-empty children, 
then $\pureconfig_v(\cT)$ is homeomorphic to one of the following spaces:
\begin{align*}
X_n& := \{ a_1,\dots, a_n\in \IC \;\vert\; a_i\neq a_j \text{ for } i\neq j\},\\
X^*_{n,N}& := \{ a_1,\dots, a_n\in \IC \;\vert\; a_i\neq 0, \; a_i\neq \zeta a_j \text{ for } i\neq j, \zeta^N=1\}.
\end{align*} 
\end{thm}

Compared to the untwisted case  
already studied in \cite{doucot2022local, doucotrembado2023topology}, 
the second factors $X^*_{n,N}$ are new: they are hyperplane complements whose associated hyperplane arrangements are not the complexification of some real hyperplane arrangement. Interestingly, the corresponding braid groups have been studied in \cite{BMR1998complex}, 
and the corresponding Weyl groups are the generalised symmetric 
groups  (see below and  Rmk. \ref{rmk: bmr etc}).

See Example \ref{eg: bigexample} p.\pageref{eg: bigexample} for 
a somewhat involved example of how the configuration space may thus be read off from the fission tree.

As a consequence, the pure local wild mapping class groups (i.e. the fundamental groups of these configuration spaces) factorise as products of the pure braid groups associated to these hyperplane arrangements:

\begin{cor}
Let $Q$ be a \mit{} and let $\mathcal{T^\flat}$ be its truncated fission tree. 
We have
\[
\purelwmcg(Q)\cong \prod_{v\in \IV^\flat} \purelwmcg_v(\cT),
\]
with $\purelwmcg_v(\cT) := \pi_1(\pureconfig_v(\cT))$ the pure braid group associated to the hyperplane complement 
$\pureconfig_v(\cT)$.
\end{cor}

In a similar way as in \cite{doucotrembado2023topology}, passing from these pure local wild mapping class groups to the \nonpure{} local wild mapping class groups  (i.e. forgetting the order of the exponential factors) involves considering the automorphisms of the fission tree, leading to a finite group, the 
{\em Weyl group} $W(\cT)$ of the tree, an extension of $\Aut(\cT)$. 

\begin{thm}

Let $\Theta=[Q]$ be an irregular class associated to the \mit{} $Q$ with fission tree $\mathcal{T}$. Then the \nonpure{} local wild mapping class group $\fulllwmcg(\Theta)$ is an extension of the Weyl group $W(\mathcal{T})$ of the  fission tree  by the pure wild mapping class group $\purelwmcg(\Theta)$, i.e. we have a short exact sequence 

\[
1 \longrightarrow \purelwmcg(Q) \longrightarrow \fulllwmcg(\Theta) \longrightarrow W(\mathcal{T}) \longrightarrow 1.
\]
\end{thm}

In some simple examples related to the configuration spaces $X^*_{n,N}$ (see Example \ref{eq: getgensymgp}) 
the Weyl group $W(\cT)$ is isomorphic to the generalised symmetric group  $S(N,n)$, which in turn is isomorphic to 
the complex reflection group $G(N,1,n)$ in the Shephard--Todd list 
\cite{Shep-Todd}. Thus we 
have a modular interpretation (in $2d$ gauge theory) 
of an infinite family of complex reflection groups and their braid groups.
(This is parallel to the appearance \cite{bafi} of the 
$G$-braid groups in $2d$ gauge theory for any complex reductive group $G$ in the untwisted case.)

As an application of these local results we are able to 
make a global statement,  
and write 
down the dimension of the (global) 
moduli spaces $\mathfrak{M}_{g,\mathbf{F}}$ 
of rank $n$, trace-free
wild Riemann surfaces for any $n$, 
a generalisation of Riemann's count 
$3g-3$   of 
the number of moduli of a compact genus $g$ Riemann surface.  
Indeed in the trace-free case ($G=\SL_n(\IC)$) we expect 
$\mathfrak{M}_{g,\mathbf{F}}$   to be Deligne--Mumford if  
$2g-2+\sum \nu(\cT)> 0$ 
where $\nu(\cT)=1+\Katz(\cT)$
and we sum over all the trees $\cT$ in the forest 
$\mathbf{F}$, and then 
its  dimension is 
\beq
\dim(\mathfrak{M}_{g,\mathbf{F}}) = 3g-3+\sum \mu(\cT)
\eeq
where $\mu(\cT)$ is the moduli number
of the fission tree $\cT$,  
from Defn. \ref{defn: moduli number} (equal to $1$ plus the dimension of the configuration space, where the $1$ corresponds to the pole position).
In the tame case with $m$ marked points, this specialises to the familiar formula:
$$\dim(\mathfrak{M}_{g,m}) = 3g-3+m.$$

\subsection{Layout of the paper}

The next section will review the notion of 
Stokes circles and irregular classes in more detail, before defining various convenient flavours of irregular types, adding certain ordering data to an irregular class.
Then the notion of admissible deformation will be reviewed leading to the initial definition of the configuration spaces. 
Section \ref{sn: classn} then gives the classification of admissible deformations, in terms of level data and fission data leading up to the precise definition of fission trees. This then yields the most convenient definition of the configuration spaces in terms of realizations of fission trees 
(Thms. \ref{thm: charn of realzns}, \ref{thm: configs from realzns}).
Then we deduce the product decompositions 
(Cor. \ref{cor: config prod decomp})
and the classification of admissible deformations in terms of fission trees (Cor. \ref{cor: adm defm types condns}).
Then \S\ref{ssn: skeleta} deduces the global result 
involving fission forests/topological skeleta 
$(g,\mathbf{F})$, and finally \S\ref{sn: lpwmcgs} deduces the results on the local wild mapping class groups and establishes the link to the braid groups of the complex reflection groups $G(N,1,n)$. Some of the possible future projects are discussed in the last section.

\section{General setting}
\subsection{Twisted irregular types} \ 

\vspace{-.34cm}
\subsubsection{The exponential local system} 
The formal data of irregular connections in the twisted case can be formulated geometrically in terms of the so-called exponential local system, which we now briefly recall (see \cite{tops} for more details). In summary the idea is to see the exponents $q$ of the exponential factors $\exp(q)$
(occurring in formal solutions of connections, controlling their essentially singular behaviour)
as sections of an intrinsic 
covering space $\pi\colon\cI\to \partial$
(i.e. a local system of sets)
on the circle of directions around the singularity. 
(We will often abuse language and refer to the exponents $q$ themselves
as the exponential factors.)
The connected components of $\cI$ are the 
Stokes circles.
The basic ``extra modular parameters'', the irregular class (that we eventually want to deform), 
is a finite multiset in the set $\cS=\pi_0(\cI)$ of Stokes circles---this amounts to choosing a finite number of Galois orbits of exponents $q$, each with a multiplicity $\ge 1$.
Note that we use the word ``twisted'' 
to refer to the case where one of the Stokes circles is a nontrivial cover of the circle of directions $\d$
as in \cite{twcv} (similarly to the theory of twisted loop groups), or equivalently that one of the exponents $q$ involves a root of the local coordinate (this is sometimes  referred to as the case with ``ramified formal normal form'', and should not be confused with the term ``ramified connection'', meaning any type of singular connection).

Let $\Sigma$ be a Riemann surface and 
$a\in \Sigma$ a point. 
Let $\phi\colon \widehat{\Sigma}\to \Sigma$ be 
the real oriented blow-up at $a$ of $\Sigma$. The preimage $\partial := \phi^{-1}(a)$ is a circle whose points correspond to real oriented directions in $\Si$ at $a$. An open interval 
$U\subset \partial$ determines  
a germ of sector $\Sect_U$ at $a$, 
and if $d\in \d$ is a direction then $\Sect_d$ will denote the germ of an open sector spanning the direction $d$ 
(where both the opening and the radius may decrease).
Strictly speaking $\Sect_d$ is the tangential filter of open sets determined by the direction $d$, as in \cite{bbkGenTop} Ch.1 \S6 (this type of tangential filter appears on p.85 of \cite{deligne-3pts}).
In turn, functions on $\Sect_d$ (i.e. germs in the direction $d$) are defined as usual for germs of mappings with respect to a filter (\cite{bbkGenTop} p.66).

Let $z$ be a local coordinate vanishing at $a$ and write $x=z^{-1}$. The exponential local system $\Ical$ is a local system of sets (that is, a covering space) on $\partial$ whose sections are germs of holomorphic functions  on 
sectors (``Fabry functions'') that are finite sums of the form:
\[
q = \sum_{i} a_i x^{k_i},
\]
where $k_i\in \mathbb{Q}_{>0}$, and $a_i\in\IC$.
More precisely if we fix a direction $d\in \d$ and 
choose a branch of $\log(z)$
on $\Sect_d$ then the fibre $\cI_d=\pi^{-1}(d)$ 
of $\cI$ over $d$ is the set
of all such functions on $\Sect_d$, so that
\begin{equation}\label{eq: iso to ppp}
\cI_d = \left\{ q=\sum_{i} a_i x^{k_i} \right\}
\quad\cong\quad
\bigcup_{n\in \IN} x^{1/n}\IC[ x^{1/n} ]
=\bigcup_{n\in \IN} 
\IC\flp z^{1/n}\frp/\IC\flb z^{1/n}\frb 
\end{equation}
where $x^k := \exp(-k\log(z))$ on the left, and $x=1/z$ is viewed as a symbol on the right.
Thus they are ``principal parts of Puiseux series'', 
but viewed as actual functions on $\Sect_d$, 
via a choice of logarithm 
(i.e. the isomorphism from $\cI_d$ to Puiseux principal parts depends on this choice).
The intrinsic (coordinate free) definition of $\cI$ is in 
\cite{twcv} Rmk. 3, whence a point $\al\in \cI_d$ is an {\em equivalence class} of certain holomorphic functions 
$q_\al$ on $\Sect_d$.

The connected component of $\cI$ of such a local section 
$q$ is a finite order cover of the circle $\partial$; this covering circle, ``the Stokes circle $\<q\>$  of $q$'', is 
essentially the (germ near $\partial$ of the) 
Riemann surface where $q$ becomes single valued.

More precisely, let $r=\ram(q)$ be the smallest integer such that the expression $q=\sum_{i} a_i x^{k_i}$ is a polynomial in $x^{1/r}$, the \textit{ramification order} of $q$. 
The corresponding holomorphic function is multivalued, and becomes single-valued when passing to a finite cover 
$t^r=z$. Therefore, the corresponding Stokes circle, which we denote by $\cir{q}$, is an $r$--sheeted cover of $\partial$. As a topological space, it is homeomorphic to a circle, and $\Ical$ is thus a disjoint union of (an infinite number of) these Stokes circles.
Thus $\pi\colon\cI\to \d$ is a covering space and if $I=\<q\>\subset \cI$ is a connected component 
then $\pi\colon I\to \d$ is a degree $r=\ram(q)$ covering map between two circles.

There are several polynomials in $x^{1/r}$ giving rise to the same connected component $I=\cir{q}$. They correspond to the Galois orbit of $q$, under the Galois group of $I\to \partial$ which is isomorphic to $\Z/r\Z$, and are parameterised by the $r$ points of the fibre 
$I_d=\pi^{-1}(d)$ for any direction $d$.

Explicitly, if we write $q=\sum_{j=1}^{s} a_j x^{j/r}$, $r=\ram(q)$, with $a_s\neq 0$, 
then the polynomials $q_i$ such that $\cir{q_i}=\cir{q}$ are the
Galois conjugates
\[
q_i=\si^i(q)=\sum_{j=1}^s a_j \omega^{ij} x^{j/r}, 
\qquad i=0,1,\ldots,r-1,
\]
where $\omega=\exp(-{2\sqrt{-1}\pi/r})$.
The fibre $I_d$ above $d$ of the cover $I\to \partial$ is equal to the set of germs of functions $q_0,\dots, q_{r-1}$. 
The monodromy $\sigma\colon I_d\to I_d$ of the cover 
$I\to \partial$ is given by $\sigma(q_i)=q_{i+1}$.

The degree $s$ of $q$ as a polynomial in $x^{1/r}$ is called the  irregularity of $q$, 
which we denote by $\Irr(q)\in \IN$. 
The slope of $q$, $\slope(q) := \frac{s}{r}=\Irr(q)/\ram(q)$, is the maximal exponent present in $q$. 
If $r=1$ we say that the circle $\cir{q}$ is untwisted/unramified. We will refer to $\cir{0}$ as the \textit{tame circle}.

Here we view $\cI$ as a disjoint union of circles, and the map $\pi\colon\cI\to \d$ as a covering space with discrete fibres. Later, below, we will deform the functions $q$, thus ``remembering'' the complex vector space structure of the fibres of $\pi$.

\subsubsection{Irregular classes, finite subcovers and levels}

In this language, following \cite{twcv} Prop. 8, 
an \textit{irregular class} is a locally constant map $\Theta \colon \Ical \to \mathbb{N}$,
assigning an integer to each component of $\Ical$, equal to zero for all but a finite number of circles. It is thus constant on each component circle, i.e. corresponds to a map $\pi_0(\Ical) \to \mathbb{N}$. An irregular class can be written as a formal sum
\[
\Theta=n_1\cir{q_1}+ \dots+ n_m \cir{q_m},
\]
where $n_1,\dots,n_m > 0$ are integers. 
Thus an irregular class $\Th$ is 
just a finite multiset of Stokes circles, or in concrete terms a finite multiset of Galois orbits of exponential factors.
(A more general definition of irregular classes, that works for other structure groups, is in \cite{twcv} \S3.5.)
The {\em rank} of the  irregular class $\Th$ is 
the integer $\Rank(\Theta) := \sum_i n_i \ram(q_i)$.

The (total) ramification $\ram(\Th)$ of an irregular class $\Th=\sum n_iI_i$ is the lowest common multiple of the ramifications $\ram(I_i)$ of the active circles $I_i$ in $\Th$.

By the 
formal meromorphic classification of meromorphic connections
(Fabry, Cope, Hukuhara, Turrittin, Levelt, Jurkat 
\cite[Thm. II]{BJLformal}, Deligne \cite[Thm. IV.2.3]{malg-book})
any connection on a rank $n$ vector bundle on the formal punctured disk determines an irregular class of rank $n$, taking the Galois orbits of the exponents of the exponential factors $e^q$ that occur, repeated according to their multiplicity. 
For example a 
regular singular connection has class $n\<0\>$ just involving the tame circle.

A ``finite subcover'' is a subset $I\subset \mathcal{I}$ such that $I \to \partial$ is a finite cover, i.e. it is finite set of Stokes circles. 
An irregular class determines a finite subcover consisting of the active exponents $I = \Theta^{-1}(\mathbb{N}_{>0})$. Explicitly, if $\Theta=\sum_i n_i I_i$, then $I=\bigsqcup_i I_i$. Thus an irregular class corresponds to the data of a finite subcover $I\subset \Ical$, together with a positive integer $n_i$ for each connected component.

Any rank $n$ irregular class $\Th$ determines another irregular class $\End(\Th)$ of rank $n^2$ (as in \cite{tops} pp.71-72).
The non-zero slopes of the circles in $\End(\Th)$ are the {\em levels} of $\Th$ (\cite{DMRci} p.73, \cite{L-R94} p.858, \cite{BBRS91} (5.2), \cite{MR91}).
Thus
\beq\label{eq: levels}
\Levels(\Th)=\{ \slope(q_\al-q_\be)\st \al,\be\in I_d\}\setminus \{0\}\subset \IQ_{>0}
\eeq
where $I\subset \cI$ is the finite subcover underlying $\Th$ and $I_d$ is any fibre of $I$.
Note that the existence of  irregular classes with multiple levels means there are connections whose formal fundamental solutions are not $k$-summable for any $k$ (and in particular  not Borel summable), and this fact led to the theory of {\em multisummation} (cf. \cite{BBRS91} \S5).

\subsubsection{Irregular types}

In the untwisted case, a distinction was made between irregular types and irregular classes, the difference being that for irregular classes the exponential factors are unordered, whereas for irregular types the order of the exponential factors matters\footnote{In the 
untwisted case an irregular class is the same thing 
as the ``bare irregular type'' determined by an 
irregular type in the sense of \cite{gbs} Rmk. 10.6, 
with ``irregular type'' as defined in \cite{gbs} Defn. 7.1.
(cf. also \cite{smid} Defn. 2.4 in the generic case, which gave a coordinate free approach to \cite{JMU81}).}.
We will now do the same for the twisted case. 
As a first step (as in \cite{tops} \S5.3)
we can just choose an ordering of the circles in an irregular class: 
\begin{defn}
An ``irregular type'' of rank $n$ 
is an ordered list 
$$[(n_1,I_1),(n_2,I_2), \ldots,(n_m,I_m)]$$ 
of distinct Stokes circles $I_i\subset \cI$ each with a
multiplicity $n_i\ge 1$, such that $\sum n_i\Ram(I_i)=n$.
\end{defn}

However here it will be more convenient to work with ordered lists of exponential factors.
Thus in the rest of the article we fix once and for all a direction $d\in \partial$, local coordinate $z=x^{-1}$ and a choice of logarithm around $d$, so that a section of the exponential local system 
$\mathcal{I}$ around the direction $d$ is identified with a 
{Puiseux principal part} $q\in x^{1/r}\C[x^{1/r}]$ for some $r$, as in \eqref{eq: iso to ppp}.

\begin{defn}
A ``full irregular type'' of rank $n$ is a Galois closed ordered list $Q=(q_1,\dots, q_n)$ of 
not necessarily distinct polynomials 
$q_i\in x^{1/r_i}\C[x^{1/r_i}]$ for some 
$r_i\in \IN_{> 0}$. 
\end{defn}

Asking for the list $(q_1,\dots, q_n)$ to be Galois closed is equivalent to ask for it to be closed under the monodromy $\si$ of the exponential local system $\Ical$. Explicitly, the list $Q=(q_1,\dots, q_n)$ is Galois closed if there exists a permutation $\wh\sigma\in \Sym_n$ 
such that 
$\si(q_i)=q_{\wh\sigma(i)}$ for $i=1,\dots, n$.

More intrinsically, given a rank $n$ irregular class $\Th$ and a direction $d$ then $\Th$ determines a length $n$ multiset in $\cI_d$, and the full irregular types determining $\Th$ correspond exactly to the $n!$ possible orderings of this multiset. 
If $Q$ is a full irregular type, let $I_1, \dots, I_m$ be the set of \textit{distinct} Galois orbits of the elements of the list $(q_1,\dots, q_n)$ and $n_i$, $i=1,\dots,m$ the multiplicity of $I_i$, i.e. the number of times that each element of $(I_i)_d$ appears in the list. Then the irregular class associated to $Q$ is 
$\Theta=n_1 I_1 +\dots n_m I_m$, and we will write 
$\Th=[Q]$ for the class of $Q$. 
In particular $n=\Rank(\Th)=\sum n_i\Ram(I_i)$.
In these terms, for any irregular class $\Theta$, the set of full irregular types determining $\Th$ corresponds to all possible orderings of the elements of the list $(q_1,\dots, q_n)$ of one full irregular type determining $\Theta$.

It will also be useful to introduce a variant of these definitions, and consider a specific subset of full irregular types, to account for the fact that the Galois orbits are already naturally cyclically ordered.

\begin{defn}\label{defn: pit}
A ``\mit{}'' is an ordered list 
$$Q=[(n_1,q_1),\ldots,(n_m,q_m)]$$ 
where $n_i\in \IN_{>0}$,
and the $q_i$ are 
Puiseux principal parts lying in distinct Galois orbits. 
\end{defn}

We identify a rank $n$ \mit{} $Q$ as a full irregular type 
as follows: 
$$
Q=(
\underbrace{q_{1}, \dots, q_{1}}_{n_1 \text{ times}},
\underbrace{\si(q_{1}), \dots, \si(q_{1})}_{n_1 \text{ times}},
\dots, 
\underbrace{\si^{r_1-1}(q_{1}), \dots, 
 \si^{r_1-1}(q_{1})}_{n_1 \text{ times}},$$
$$\qquad\qquad\qquad\qquad\dots,\dots,\qquad\qquad \dots, \dots, \dots,
\underbrace{\si^{r_m-1}(q_{m}), \dots, 
\si^{r_m-1}(q_{m})}_{n_m \text{ times}}
)
$$

\noindent
where $r_i=\ram(q_i)$. 
Note that (more intrinsically) a rank $n$ \mit{} 
is equivalent to an ordered list 
$$[(n_1,p_1,I_1),\ldots,(n_m,p_m,I_m)]$$ 
where $n_i\in \IN_{>0}$, the $I_i$ are 
distinct Stokes circles, $p_i\in (I_i)_d$
is a point of $I_i$ lying over $d$,
and $n=\sum n_i\ram(I_i)$.
The correspondence is given by taking $q_i$ to be the 
Puiseux principal part determined by 
the point $p_i\in (I_i)_d$ via the logarithm choice.

Finally observe that the notion of pointed irregular type introduces
extra discrete invariants that we do not 
care about
(for example $\slope(q_1-q_2)$ may vary if $q_1$ is moved in its Galois orbit). 
To avoid this we define a notion of ``compatibility'' between the chosen exponential factors in different Galois orbits. 
For any $k\in \IQ$ let 
\beq\label{eq: truncation}
\tau_k\colon\cI_d\to \cI_d, \qquad q=\sum a_ix^{k_i}\longmapsto 
\tau_k(q)  :=  \sum_{k_i\ge k} a_ix^{k_i}
\eeq
be the truncation map, discarding all monomials of slope $<k$.
\begin{defn}\label{defn: compatible pit}
A \mit{} $Q=[(n_1,q_1),\dots, (n_m,q_m)]$ is \textit{compatible} if for each possible exponent $k\in \IQ_{>0}$, and any indices $i,j$, 
$$\cir{\tau_k({q_i})}=\cir{\tau_k({q_j})} \quad\Longrightarrow\quad 
\tau_k({q_i})=\tau_k({q_j}).$$
In other words: if the truncations are in the same Galois orbits, then the truncations are equal.
\end{defn}

It is straightforward to see that for any irregular class $\Theta$, a compatible (pointed) irregular type $Q$ exists with class $\Th=[Q]$. 
Up to isomorphism the configuration space 
$\pureconfig(Q)$ will not depend on the choice of irregular type with irregular class $\Theta$, so we may assume without loss of generality that the pointed irregular types we are considering are compatible.

\subsubsection{Pullback to untwisted case}

The notions of (twisted) irregular classes and irregular types can easily 
be related to the corresponding untwisted notions, by passing to a finite cyclic cover. Explicitly, let $Q=(q_1,\dots, q_n)$ be any full irregular type. 
Let $r$ be an integer multiple of $\Ram(Q)$
so that $\ram(q_i)$ divides $r$ for all $i$.
Introduce the variable $t$ such that $t^r=x$ (so $t^{-1}$ is a coordinate on a cyclic $r$-fold ramified cover). 
Let $\wh{q}_i\in t\C[t]$ be 
$q_i$ seen as a polynomial in $t$. 
Then $\wh{Q} := \diag(\wh{q}_1,\dots,\wh{q}_n)$ is an untwisted irregular type associated to $\Theta$. Its (untwisted) 
irregular class $\wh{\Theta}=[\wh Q]$ only depends on $\Theta$
and is simply the pullback. 
Notice that the irregular class $\wh{\Theta}$ is invariant under the action of $\Z/r\Z$ on the set of untwisted irregular classes obtained by replacing all polynomials 
$\wh{q}_i(t)\in \C[t]$ by $\wh{q}_i(e^{2\sqrt{-1}k\pi/r}t)$, for any integer $k$: we say that the untwisted irregular class $\wh{\Theta}$ is $r$-Galois closed. Conversely, if $\wh{Q}$ is an $r$--Galois closed untwisted irregular type, it defines a (twisted) irregular type $Q$ such that the ramification orders of all exponential factors divide $r$. The (twisted) irregular class of $Q$ only depends on the (untwisted) irregular class of 
$\wh{Q}$.

\subsection{Admissible deformations}

The notion of admissible deformations was
defined in \cite{gbs} for arbitrary 
untwisted meromorphic connections in the context of 
any reductive group $G$, extending 
the generic case in \cite{JMU81, Mal-imd2long}  for $\GL_n(\IC)$ 
and in \cite{bafi} for other $G$. 
It can be extended
to the twisted setting simply by saying that a family of irregular classes is an admissible deformation if and only if some (and hence any) cyclic pullback to the untwisted case is 
an admissible deformation.  
In more detail this works out as follows.

Fix $\cI$ as above and let 
$\IB$ be a connected complex manifold.
Choose a rank $n$ irregular class 
$\Th_b$ on $\cI$ for each $b\in \IB$, thus defining a (set theoretic) map
$$\phi\colon\IB\longrightarrow \operatorname{IC}_n(\cI), \qquad b\longmapsto \Th_b$$
to the set  $\operatorname{IC}_n(\cI)$ of rank $n$ irregular classes, i.e. length $n$ multisets in $\pi_0(\cI)$. 
We will define when this collection of classes 
is a (holomorphic) admissible deformation.

Note that since the rank $n$ is fixed the total 
ramification $\ram(\Th_b)$ is uniformly bounded\footnote{
In fact since $n=\sum n_i\ram(I_i)$ it is bounded by the largest possible lowest common multiple of the elements of any integer partition of $n$; this is known as Landau's function $g(n)$, whose first 10000 values are listed at 
\url{https://oeis.org/A000793}, for example 
$g(5)=6=\lcm(2,3), g(7)=12=\lcm(3,4)$.} 
on $\IB$, for example by $n!$. 
Thus we can choose an integer $N$ and set $x=t^N$ so that (in terms of $t$) $\Th_b$ is a family of untwisted irregular classes, i.e. a multiset in $t\IC[t]$ of length $n$, for each $b\in \IB$.

By definition (see \cite{gbs} Rmk.10.6) 
this is an admissible deformation if it can locally be represented as $\Th_b=[Q_b]$ in terms of an admissible family of untwisted irregular types 
$Q_b=(q_1,\ldots, q_n)$ with $q_i\in t\IC[t]$ 
dependent on $b$.
Finally (by \cite{gbs} Defn. 10.1) this is a (holomorphic) admissible deformation if each $q_i$ varies holomorphically with $b$ and the degree of the polynomial 
$q_i-q_j\in t\IC[t]$ is constant (independent of $b$) 
for each $i\neq j$ (the degree is an integer $\ge 0$).
Similarly one can define smooth admissible deformations etc by allowing  the coefficients of the 
$q_i$ to vary smoothly rather than holomorphically etc. 
This leads to the following more direct definition, by noting that the slope multiplied by $N$ gives
 the degree in $t$ when pulled back, upstairs.

\begin{defn}\label{defn: hol adm defmn}

\noindent
$\bullet$ A holomorphic family $Q_b=(q_1,\ldots,q_n)$ of full irregular types (with $q_i\in \cI_d$) 
is an admissible deformation if 
\begin{equation}\label{eq: adm condn}
\boxed {\vphantom{\Bigl\vert}
\slope(q_i-q_j) \text{ is independent of $b$ for all $i,j$.}
}\end{equation}

\noindent
$\bullet$ 
The family $\Th_b$ of irregular classes is a holomorphic admissible deformation if 
it can locally be represented as $\Th_b=[Q_b]$ for 
a holomorphic family of full irregular types 
$Q_b=(q_1,\ldots,q_n)$ with $q_i\in \cI_d$, 
varying admissibly, i.e. satisfying \eqref{eq: adm condn}.
\end{defn}

If $\Theta$ and $\Theta'$ are two rank $n$ irregular classes, we say that $\Th'$ is an admissible deformation of $\Th$ if there exists an admissible deformation $(\Theta_b)$ indexed by some connected manifold $\IB$ equipped with two points $b_1,b_2\in \IB$ such that such that 
$\Theta_{b_1}=\Th$ and $\Theta_{b_2}=\Th'$.
Similarly, if $Q$ and $Q'$ are two rank $n$ full irregular types, we will say that $Q'$ is an admissible deformation of $Q$, and we write $Q\simeq Q'$, if 
they are two values of an admissible family of full irregular types.
A continuity argument shows the Galois orbits in a full irregular type do not change under a holomorphic admissible deformation,
i.e. the same permutation 
$\wh \si$ works throughout the deformation (in particular the ramification indices of the Stokes circles are constant).

\begin{eg} Let us consider the holomorphic family of  exponential factors
$q(b)=x^{1/2}+x^{1/3}+bx^{1/6}$, which has ramification $6$ for any $b\in \mathbb{C}$. The first few Galois conjugates of $q$ are: 
\begin{align*}
q=q_0(b)&=x^{3/6}+x^{2/6}+bx^{1/6}\\
\si(q)=q_1(b)&=\varepsilon^3 x^{3/6}+\varepsilon^2 x^{2/6}+b\varepsilon x^{1/6}\\
\si^2(q)=q_2(b)&= x^{3/6}+\varepsilon^4 x^{2/6}+b\varepsilon^2 x^{1/6}
\end{align*}
\noindent 
where $\varepsilon=\exp(-\pi i/3)$.
Considering $\slope(q_i-q_j)$ for $i,j=0,\ldots,5$ shows that $\Th_b=\<q\>$ is an admissible deformation over 
$\IB=\IC$.
Observe that for $b=0$ we have $\ram(q_0-q_2)=3$, 
but it is $6$ for $b\neq 0$
so not everything behaves continuously.
\end{eg}

\subsubsection{Numerical equivalence of irregular types}

We will try to guess a simple numerical criterion 
for (pointed) irregular types to be admissible deformations
of each other.
To this end consider the following relation.

\begin{defn}
Let $Q=[(n_1,q_1),\dots, (n_m,q_m)]$ be a \mit{}. 
If $Q'=[(n'_1,q'_1),\dots, (n'_p,q'_p)]$ with each
$q_i'\in \cI_d$ a Puiseux polar part, 
then we say that $Q,Q'$ are ``numerically equivalent'', and write
\beq\label{eq: numericalsim}
Q'\sim Q,
\eeq
if $p=m$, $n'_i=n_i$ ($i=1,\ldots,m$), and 
\beq\label{eq: numericalequiv}
\slope(\si^k (q'_{i})-\si^l (q'_{j}))=
\slope(\si^k(q_{i})-\si^l(q_{j})),
\eeq
for  all $i,j,k,l$ with 
$1\leq i,j\leq m$, $0\leq k\leq \ram(q_i)$, $0\leq l\leq \ram(q_j)$, where $\si$ is the Galois action. 
\end{defn}

\begin{lem}
If $Q'\sim Q$ as above, then $Q'$ is a \mit{} and moreover 
$\ram(q'_i)=\ram(q_i)$ for all $i$. 
\end{lem}
\pf
Taking $j=i$, $k=\ram(q_i)$, $l=0$ shows that 
$\ram(q'_i)\le\ram(q_i)$.
Thus $\ram(q'_i)=\ram(q_i)$
since if $\ram(q'_i)<\ram(q_i)$
then there would be some identification
amongst the list $\si^k(q'_i)$, $k=1,2,\ldots,\ram(q_i)$,
but this is not possible as the differences of the 
slopes matches that of the list
$\si^k(q_i)$, $k=1,2,\ldots$.
Then the fact that $Q'$ is a \mit{}, i.e. its $m$ Galois 
orbits are distinct, follows from the fact that $Q$ is a \mit{}, so none of the slopes between two Galois orbits vanishes.
\epf

This implies $\sim$ is an equivalence relation when restricted to pairs of \mits{}.
We will eventually see (Cor. \ref{cor: trees and adm defmns of types }) that for compatible \mits{}  it is the same as the relation given by admissible deformation.

Thus it seems we should consider the simple numerical  condition  
\eqref{eq: numericalequiv} applied  blindly to lists of 
Puiseux polar parts. This leads to the following configuration spaces.

\begin{remark}
Note that if we just impose that $Q'=[(n_1,q'_1),\dots, (n_m,q'_m)]$ 
is a rank $n$ \mit{} and the apparently weaker condition that 
\eqref{eq: numericalequiv} holds just for 
$1\leq i,j\leq m$, $0\leq k\leq \ram(q_i)-1$, $0\leq l\leq \ram(q_j)-1$,
then it follows that $Q'\sim Q$
(because this implies $\ram(q_i')\ge \ram(q_i)$ and then the condition to have rank $n$ implies $\ram(q_i')=\ram(q_i)$ for all $i$).
\end{remark}

\subsubsection{Configuration spaces}

We will define a \textit{configuration space} for each given (pointed) irregular type 
$Q=[(n_1,q_1),\dots, (n_m,q_m)]$, 
and later see it contains all the admissible deformations with bounded slope. 
Let $r=\Ram(Q) = \lcm \{\ram(q_i)\}$  
be the total ramification of $Q$ and let 
$$K=\Katz(Q) := \max(\slope(q_1),\ldots,\slope(q_m))$$
be the largest slope, which is 
essentially the Poincar\'e--Katz rank of $Q$ 
(cf. Poincar\'e \cite{poincare1886} p.305, Katz 
\cite{katz1970} 11.9.7).
Thus  all the $q_i$ can be expressed as polynomials 
in $t :=  x^{1/r}$ of degree at most $s :=  rK$.

Clearly any \mit{} with the same number $m$ of terms, 
the same multiplicities and ramifications, and that has Poincar\'e--Katz rank $\le K$, will be of the form:
\beq\label{eq: naivecoeffs}
Q_\mathbf{a} := 
[
(n_1,\sum_{j=1}^s a_{1,j}t^{j}), \dots,
(n_m ,\sum_{j=1}^s a_{m,j}t^{j})],\qquad t=x^{1/r}
\eeq 
for some unique collection of coefficients 
$\mathbf{a}=(a_{i,j})\in \C^{ms}.$
This motivates the following definition.

\begin{defn}\label{defn: config space}
Suppose $Q=[(n_1,q_1),\dots, (n_m,q_m)]$ is a \mit{}
and $r=\ram(Q), K=\Katz(Q), s=rK$. 
The configuration space of $Q$ with bounded Poincar\'e--Katz rank is the topological space  $\pureconfig(Q)$  defined by
\begin{equation}
\pureconfig(Q) := \left\{ \mathbf{a}=(a_{i,j})\in \C^{ms}  \st Q_{\mathbf{a}} \sim Q \right\},
\end{equation}
with its topology being the one induced from the usual topology of $\C^{ms}$, where $\sim$ is from \eqref{eq: numericalsim}.
\end{defn}

We will show below (Cor. \ref{cor: config space as fine modspace }) that $\pureconfig(Q)$ is a fine moduli space of all admissible deformations (with Poincar\'e--Katz rank $\le \Katz(Q)$) of the pointed irregular type $Q$. 
In the remainder of the article, our goal will then be to explicitly describe $\pureconfig(Q)$ and compute its fundamental group.  
The restriction about having bounded slopes is for the sake of convenience, since it allows us to deal with a finite number of coefficients. As was already the case \cite{doucot2022local, doucotrembado2023topology} for the untwisted situation, this entails no loss of generality as far as the topology of $\pureconfig(Q)$ is concerned: up to homotopy equivalence $\pureconfig(Q)$ does not change if we allow for coefficients associated to higher exponents.

Notice that if $Q_1$ and $Q_2$ are two \mits{} corresponding to the same irregular class $\Theta$, the spaces $\pureconfig(Q_1)$ and $\pureconfig(Q_2)$ are homeomorphic, an homeomorphism being given by permuting the active circles and shifting cyclically the distinguished representative of each Galois orbit by the appropriate amount. With a slight abuse of language, we may thus speak of the configuration  space 
$\pureconfig(\Theta)$ that is well--defined up to homeomorphism.

Similarly we define a configuration space of trace-free \mits{}.
First define the {\em trace} of a full irregular type $Q=(q_1,\ldots,q_n)$ 
to be  $\Tr(Q)=\sum_1^n q_i\in \cI_d$. 

\begin{defn}
Suppose $Q=[(n_1,q_1),\dots, (n_m,q_m)]$ is a \mit{}
and $r=\ram(Q), K=\Katz(Q), s=rK$. 
The traceless (or {\em special}) configuration space of $Q$ is the topological space  
$\mathbf{SB}(Q)$  defined by
\begin{equation}\label{eq: SB}
\mathbf{SB}(Q) := \left\{ \mathbf{a}=(a_{i,j})\in \C^{ms}  \st Q_{\mathbf{a}} \sim Q, \Tr(Q_{\mathbf{a}})=0 \right\},
\end{equation}
with its topology being the one induced from the usual topology of $\C^{ms}$, where $\sim$ is from \eqref{eq: numericalsim}.
\end{defn}

If $Q=[(n_1, q)]$ just has one Galois orbit then 
$\Tr(Q)=n_1\sum_{i=1}^{\ram(q)}\si^i(q)$. 
In turn, since roots of unity sum to zero, this equals 
$\Tr(Q)=n_1\Ram(q)\pi_{un}(q)$ where 
$\pi_{un}\colon\cI_d\to x\IC[x]$ is the linear map picking out the unramified monomials in $q$, so that $\pi_{un}(x^{k}) = x^k$ if $k\in \IN$ and 
$\pi_{un}(x^{k}) = 0$ otherwise.
It follows that the trace of any irregular type 
lies in the unramified part $x\IC[x]\subset \cI_d$.
Further there is a projection $\pr$:
$$Q=(q_1,\ldots,q_n)\mapsto \pr(Q)=Q-\frac{1}{n}\Tr(Q) = (q_1-\frac{1}{n}\Tr(Q),\ldots,q_n-\frac{1}{n}\Tr(Q))$$
mapping any full irregular type to a trace-free irregular type.
In particular it makes no difference if we replace $Q$ by its trace-free projection in the definition \eqref{eq: SB}, and there are maps
\beq
\pureconfig(Q)\ \onto\ \mathbf{SB}(Q)\ \into\ \pureconfig(Q)
\eeq
where the first map is $\pr$ and the second is the natural inclusion.
We will see below (Cor. \ref{cor: config space as fine modspace }) that $\mathbf{SB}(Q)$ is a fine moduli space of all trace-free admissible deformations of the pointed irregular type $\pr(Q)$.
We will also show that $\mathbf{SB}(Q)$ and $\pureconfig(Q)$ are finite-dimensional complex algebraic 
manifolds (Zariski open in a complex vector space). 
Admitting this   temporarily, one can already 
observe that the dimensions will differ
 by the integer part of the Poincar\'e--Katz rank  and they are homotopy equivalent:
\begin{lem}
For any pointed irregular type $Q$, the 
configuration spaces $\mathbf{SB}(Q)$, $\pureconfig(Q)$
are homotopy equivalent, and
$$\dim(\mathbf{SB}(Q)) = \dim(\pureconfig(Q))-\floor{\Katz(Q)}.$$
\end{lem}
\pf
Two elements are in the fibre of the map 
$\pr\colon\pureconfig(Q)\onto\mathbf{SB}(Q)$
if and only if they differ by the operation 
$(q_1,\ldots,q_n)\mapsto (q_1-q,\ldots,q_n-q)$
for some $q\in x\IC[x]$ of slope $\le K$.
The dimension of the space of such polynomials $q$ is 
$\floor{\Katz(Q)}$, and this gives a retraction onto 
$\mathbf{SB}(Q)$.
\epf

\begin{remark}
Isomonodromic deformations  of a special class of 
twisted irregular connections
were 
considered in \cite{bertola-mo},
under a genericity condition 
(so that the sizes of the Galois orbits 
are controlled by the 
Jordan blocks of the leading coefficient). 
The relation between  our general admissibility condition 
and the Lidskii conditions in \cite{bertola-mo}, 
specific to their setting, 
are not immediately clear to us.
\end{remark}

\section{Classification of admissible deformations}\label{sn: classn}

Since an essential difference in the twisted case compared to the untwisted one is that one has to consider differences between different branches of the same exponential factor, it is worth investigating first what the admissible deformations are in the case of an irregular type corresponding to an irregular class with only one active circle.

\subsection{A single Stokes circle} 
Suppose $I=\<q\>\subset \cI$ is a single Stokes circle.
Recall from \eqref{eq: levels}  that the {\em levels} of $I$ are the non-zero slopes of $\End(I)$, so that, for any $d\in \partial$:
$$\levels(I)=\{\slope(q_\al-q_\be)\st \al,\be\in I_d\}\setminus\{0\}$$
where $q_\al\colon\Sect_d\to \IC$ is the function determined by 
$\al\in I_d\subset \cI_d$. 
The set $\levels(I)$ 
is a finite, possibly empty, subset of $\IQ_{>0}$.
Suppose there are $m$ levels and write
$$\levels(I) = (k_1>k_2>\cdots >k_m)\subset \IQ_{>0}.$$

The key classification statement is as follows.

\begin{prop}\label{prop: one circle classn}

a) Two Stokes circles $I,J\subset \cI$ are admissible deformations of each other if and only if $\levels(I)=\levels(J)\subset \IQ$.

b) A subset $(k_1>k_2>\cdots >k_m)\subset \IQ_{>0}$ is the set of levels of some circle $I\subset \cI$ if and only if 
\begin{equation}\label{eq: condition on levels of a circle}
\text{$k_1,k_2,\ldots,k_m$ have strictly increasing common denominators $>1$.}
\end{equation}
In other words if $d_i$ is the denominator of $k_i$ (in lowest terms) and 
\beq
r_i\text{ is the lowest common multiple of }d_1,d_2,\ldots,d_i
\eeq
for each $i$ (so that $r_i\!\st\! r_{i+1}$), then $1<r_1<r_2<\cdots <r_m$.
\end{prop}

\pf
Consider $I=\<q\>$ and let $r=\ram(q)$.
Choose a local coordinate $z$ vanishing at $0$, set $x=1/z$
and suppose $x=t^r$.
Then $q=\sum_{i=1}^n \al_i t^{n_i}$ is a polynomial in $t$
with each $\al_i$ non-zero and $n_1>n_2>\cdots>n_n\subset \IN$.
Let $r_0=1$ and let
$$r_1<r_2<\cdots<r_m=r$$ 
be the set of distinct leading common denominators $>1$ that occur, 
i.e. the distinct numbers $>1$ in the set
$\{\ram(\al_1 t^{n_1}+\cdots \al_i t^{n_i})\st i=1,2,\ldots,n\}$. 
Recall that if $b_i$ is the denominator 
of $n_i/r=\slope(\al_it^{n_i})$,
then 
$\ram(\al_1 t^{n_1}+\cdots \al_i t^{n_i})=\lcm(b_1,b_2,\ldots,b_i)$.  
Finally, for $i=1,\ldots,m$,
let $k_i\in \IQ_{>0}$ be the largest exponent 
such that the ramification is $r_i$, i.e.
$$k_i = \max\{n_j/r\st 
\ram(\al_1 t^{n_1}+\cdots + \al_j t^{n_j})=r_i\}.$$

Then we claim that the levels of $I$ are these numbers 
$k_1>k_2>\cdots >k_m$.
This is an exercise, that can be done visibly by drawing a picture, as  follows\footnote{Beware the ``naive/full fission tree'' defined here 
will be represented by a single full branch 
of the 
precise fission trees to be defined carefully 
in \S\ref{ssn: fission trees} below.
In brief,   
in the untwisted case fission trees were defined from meromorphic connections   
in the quiver modularity conjecture 
in \cite{rsode} Appx. C (i.e. Hiroe--Yamakawa's theorem \cite{hi-ya-nslcase, yamakawa20}), and
the ``naive/full fission tree'' here is defined similarly to those fission trees, after pulling back to the untwisted case.
In turn the untwisted fission trees of 
\cite{rsode} are essentially the same as the untwisted special case of the fission trees of \S\ref{ssn: fission trees}.}:

For any $k\in \IR_{\ge 0}$ define 
$q_k=\sum \al_i t^{n_i}$,
where the sum is over the indices $i$ such that $n_i/r> k$.
Thus $q_k$ is the leading piece of $q$ whose monomials have slope $> k$.
Now for each $k$ consider the finite set
\beq\label{eq: gal orbs}
N_k := \{q_k(t),q_k(\ze t),\ldots,q_k(\ze^{r-1}t)\}\subset t\IC[t]
\eeq
where $\ze=\exp(2 \pi i/r)$. 
If $k\in [k_{i+1},k_{i})$ then $\abs{N_k}=r_i$
by definition (of the $r_i$ and $k_i$).
Thus as $k$ varies the sets $N_k$ define a large disjoint union of copies of intervals (i.e. $r_i$ copies of 
$[k_{i+1},k_{i})$ for each $i$).
Moreover if $k<l$ then truncation gives a map $N_k\onto N_l$, and this tells us how to glue the intervals into a tree:
there is one interval, the trunk, over $[k_1,\infty)$. We glue this to 
the $k_1$-ends of the $r_1$ intervals over $[k_2,k_1)$.
Then over $k_2$ there are $r_1$ nodes, and we glue each of them to
$(r_2/r_1)$ of the $r_2$ intervals over $[k_3,k_2)$, etc, ending up with 
the $r$ {\em leaves} of the tree $N_0$ over $0$.
Thus the nodes where the tree branches are exactly the 
points $N_{k_1}\sqcup \cdots \sqcup N_{k_m}$, 
over $k_1>\cdots >k_m$.

Now it is easy to see the $k_i$ are the levels: identify $I_d$ with the leaves $N_0$; 
thus if $i,j\in I_d$ then 
$\slope(q_i-q_j)$ is the supremum of the unique shortest path in the tree between the leaves $i$ and $j$. 
Thus the levels are exactly where the branching occurs.

In turn, by definition (see \eqref{eq: adm condn}) the admissible deformations are thus exactly those which preserve the tree, as an abstract tree 
lying over $\IR_{\ge 0}$
(we can choose an order of $N_0$ to get an irregular type).
Consider the operation of adding to $q$ a monomial of the form 
$\al x^k$ where $k_i>k>k_{i+1}$ and $r_ik\in \IN$, as $\al\in \IC$ varies.
This operation 
clearly gives an admissible deformation of $q$
since it does not change any of the Galois orbits \eqref{eq: gal orbs}.
Thus we can admissibly deform $q$ to an element of the form
$\sum_1^m \be_ix^{k_i}$ with each $\be_i\neq 0$.
In turn we can admissibly deform 
this so that each $\be_i=1$.
This implies a), since we have shown that any two 
Stokes circles with the same levels $\{k_i\}$  can be admissibly deformed into $\<\sum_1^m x^{k_i}\>$ and thus into each other (and the converse is clear).
Statement b) is also now clear: the common denominator needs to increase to get into a bigger Galois orbit. 
\epf

\begin{remark}
Note that the statement a) includes the empty set: 
any two unramified circles 
$I,J\subset \cI$ are admissible deformations of each other.
This remark underlies the theory of Baker functions 
(see \cite{ihptalk} \S2.2 and the references there).
\end{remark}

\begin{remark}
Note that the proof really shows that
any two pointed irregular types $[(n_1,q_1)], [(n_2,q_2)]$ 
(each with just one Galois orbit) 
are admissible deformations of each other if and only if 
$\Levels(q_1)=\Levels(q_2)$, and $n_1=n_2$. 
\end{remark}

For later use we will formalise the various sets of 
data in the proof as follows:

\begin{defn}\label{defn: level datum}
A ``level datum'' is a finite, possible empty, subset 
$L\subset \IQ_{>0}$ satisfying the conditions 
\eqref{eq: condition on levels of a circle} 
(so the numbers it contains are the possible levels of a single Stokes circle).
\end{defn}

Thus, as above, a level datum $L=(k_1>\cdots > k_m)$ determines ramification indices 
$$\operatorname{RI}(L) = (r_1<\cdots < r_m)\subset \IN_{>1}$$
and in particular $\Ram(L) = r_m$.
In turn, as in the proof, we can define the 
{\em inconsequential exponents} $\Inc(L)\subset \IQ_{>0}$ as:
\begin{equation}\label{eq: inc expts}
\Inc(L) = \IN_{>0} \cup 
\left((k_2,k_1)\cap \frac{1}{r_1}\IN\right)
\cup \cdots \cup
\left((k_{m},k_{m-1})\cap \frac{1}{r_{m-1}}\IN\right)\cup
\left((0,k_{m})\cap \frac{1}{r_{m}}\IN\right),
\end{equation}
and the {\em admissible exponents}:
\begin{equation}\label{eq: adm exps}
A(L)  :=  L\sqcup \Inc(L)\subset \IQ_{>0}.
\end{equation}
If $L$ is empty then $\Ram(L)=1, \operatorname{RI}(L)=\{1\}$ and $A(L)=\Inc(L)=\IN_{>0}$. 
Note that the admissible exponents can thus also be expressed as:
\begin{equation}
A(L) = 
\left((0,k_m]\cap \frac{1}{r_m}\IN\right)
\cup \cdots \cup
\left((0,k_{2}]\cap \frac{1}{r_2}\IN\right)
\cup
\left((0,k_{1}]\cap \frac{1}{r_1}\IN\right)
\cup \IN_{>0}\subset \IQ_{>0}.
\end{equation}

Thus any Stokes circle $I\subset \cI$ determines subsets 
\beq\label{eq: L in A}
L=L(I)\subset A=A(I)=A(L)\subset \IR_{>0},
\eeq
 where 
$L=\Levels(I)$, and $A$ is the set of admissible exponents.
The ramification index $\Ram(v)\in\{1,r_1,\ldots,r_m\}$
of any $v\in \IR_{\ge 0}$ is defined
to be the least common multiple of the denominators of the levels 
$k_i\ge v$, so that
\beq\label{eq: ramification indices on branch}
\Ram(v) = r_i\qquad\text{if $v\in (k_{i+1},k_i]$}.
\eeq
Note that \eqref{eq: inc expts},\eqref{eq: adm exps}
imply the number of non-integral admissible 
exponents is given by the formula: 
\beq\label{eq: non-int adm formula}
\abs{A(L)\setminus \IN} = \sum_1^m r_ik_i -\floor{r_{i-1} k_i}
\eeq
where we set $r_0=1$.

\begin{eg} Let us look at a few simple examples.
\begin{itemize}
\item Suppose $q$ has $\ram(q)=r>1$ and $\slope(q)=s/r$
with  $s$ and $r$ coprime. Then
$q=\sum_1^s a_i x^{i/r}$ with $a_s\neq 0$, 
the only level of $q$ is its slope $\frac{s}{r}$, 
$\Levels(q)=(s/r)$ and  $A(L)=\IN_{>0}\cup\frac{1}{r}\{1,\ldots,s\}$.
\item Consider $q=x^3+x^{5/2}+x^{3/2}+x^{1/3}$. 
It has ramification order $6=\lcm(2,3)$.
The corresponding list of ramification indices is 
$\operatorname{RI}(q)=(2<6)$, and the levels are 
$\Levels(q)=(5/2>1/3)$.  
In turn $\Inc(q)=\IN_{>0}\cup\{3/2, 1/2, 1/6\},$  and 
$A(L)=\IN_{>0}\cup\{5/2, 3/2, 1/2, 1/3, 1/6\}$. 
\end{itemize}
\end{eg}

\begin{remark}\label{rmk: linktosings}
Note similar (elementary) methods appear in the theory of 
curve singularities \cite{Le-planecurves}. 
The reason for such a link to curves is 
the wild nonabelian Hodge correspondence,  
between meromorphic connections and 
meromorphic Higgs bundles, 
followed by taking the spectral curve of the Higgs field. 
Indeed the dictionary in \cite{wnabh} p.180 
determines the irregular class 
from the spectral invariants at the singularity of the Higgs field.
Note that we only consider {\em part} 
of the data of the corresponding curve singularity, and not all of it,  
i.e. the  ``principal part'' of the singularity, determining the irregular class.
This reflects the fact that fission, breaking up the curve at the pole, is about the various
growth rates of the essentially singular  functions $\exp(q)$ at $a\in \Si$, and this only involves the principal part of $q$.
\end{remark}

\subsubsection{Single circle configuration spaces}

For any pair $q_1,q_2\in \cI_d$ of exponential factors
we can consider the pointed irregular types 
$Q_1=[(1,q_1)], Q_2=[(1,q_2)]$ and say that 
$q_1\sim q_2$ if $Q_1\sim Q_2$ in the sense of 
\eqref{eq: numericalsim}, which amounts to the condition: 
$$\slope(q_1-\si^i(q_1)) = \slope(q_2-\si^i(q_2))$$
for $i=0,1,\ldots,\ram(q_1)$.

\begin{cor}
\label{cor: deformation_exp_factor_iff_same_levels}
Let $q_1$ and $q_2$ be two exponential factors. 
Then $q_1\sim q_2$ in the sense of 
\eqref{eq: numericalsim} if and only if 
$\Levels(q_1)=\Levels(q_2)$
(i.e. if and only if they are admissible deformations of each other). 
\end{cor}
\pf
Clearly if $q_1\sim q_2$ then $\Levels(q_1)=\Levels(q_2)$.
Conversely if they are admissible deformations of each other then they both can be admissibly 
deformed to 
$q_0 := \sum x^{k_i}$ (where the $k_i$ are the levels), 
and so the three lists
$$\slope(q_j-\si^i(q_j)),\qquad  i=1,2,3,4,\ldots$$
of rational numbers are equal, for $j=0,1,2$ (since they remain equal under any small admissible deformation, and under a new choice of initial $q$).
\epf

This motivates the definition of the  
configuration space $\pureconfig(q) := \pureconfig(Q)$
(from Defn. \ref{defn: config space}, using $\sim$ from 
\eqref{eq: numericalsim}),  
for any exponential factor $q$, where $Q=[(1,q)]$,
since we now see $\pureconfig(q)$ 
is the set of all exponential factors 
that are admissible deformations of $q$ and 
have Poincar\'e--Katz rank (maximal slope) at most $K := \Katz (q)$. 
Thus we deduce an explicit description of the configuration spaces 
in the case of one circle:

\begin{prop}\label{prop: singcirc conf}
Let $q\in \cI_d$ be an exponential factor, and 
let $L=\Levels(q)$ be the levels of $q$. 
Then $$\pureconfig(q)\cong (\IC^*)^m\times \IC^N,\qquad
\mathbf{SB}(q)\cong (\IC^*)^m\times \IC^M$$
where $m=\abs{L}$ is the number of levels, 
$N$ is the number $\bigl\vert\Inc^\flat(L)\bigr\vert$ 
of inconsequential exponents $\le K$, and 
$M=\abs{\Inc(L)\setminus\IN}$ is the number of  
non-integral inconsequential exponents.
In particular  
$\dim \pureconfig(q)$ is the number 
$\bigl\vert A^\flat(L)\bigr\vert$ 
of admissible exponents $\le K$
and $\dim \mathbf{SB}(q)=\abs{A(L)\setminus\IN}$ 
is the number of non-integral admissible exponents, as 
given by the formula  \eqref{eq: non-int adm formula}.
\end{prop}
\pf
Given $I=\<q\>$ we consider $L(I)\subset A(I)\subset \IR_{>0}$
as in \eqref{eq: L in A}, consisting of the admissible exponents and 
the subset of levels.
We can move the coefficients parameterised by $A(I)$ arbitrarily provided 
those from $L$ remain non-zero.
The descriptions of the configuration spaces then arise a) by not going past $K$, and b) by only considering trace-free deformations.
\epf

\begin{eg} Let us come back to our previous examples:
\begin{itemize}
\item If $q=a_s x^{s/r}+\cdots + a_1x^{1/r}$, where 
$s$ and $r=\ram(q)>1$ are coprime, $a_s\neq 0$,
then its (slope bounded) admissible deformations are of the form
\[
q'=\sum_{k=1}^s b_k x^{k/r},
\]
with $b_s\in \C^*$ non-zero, and $b_1,\dots, b_{s-1}\in \IC$ arbitrary.
Removing the integral exponents leaves $s-\floor{s/r}$ coefficients,  
agreeing with formula \eqref{eq: non-int adm formula} for 
$\dim \mathbf{SB}(q)$ in this case.
\item For $q=x^3+x^{5/2}+x^{3/2}+x^{1/3}$, the set of levels is $L(q)=\{5/2,1/3\}$, the set of admissible exponents $\le 3$ is 
$\{3, 5/2, 2, 3/2, 1, 1/2,1/3, 1/6\}$, 
so the (slope bounded) admissible deformations of $q$ are of the form
\[
q'=\al x^3+a x^{5/2}+b x^2+c x^{3/2}+d x+ e x^{1/2}+f x^{1/3}+g x^{1/6},
\]
with $\al, a,b,c,d,e,f,g\in \IC$, with $a, f$ non-zero, and the other coefficients arbitrary.
The trace-free projection of $q$ is $\pr(q)=q-x^3$ and the 
trace-free admissible deformations of this are as above, but with 
$\al=b=d=0$. This deformation space has dimension $5$, agreeing with 
 \eqref{eq: non-int adm formula}.
\end{itemize}
\end{eg}

\begin{remark}
Note that if we change coordinates the subsets
$L(I)\subset A(I)\subset \IR_{>0}$
attached to any Stokes circle $I$ (as in \eqref{eq: L in A})
do not change.
(Similarly for the fission trees to be defined below.)
\end{remark}

\subsection{The case of two Stokes circles}

As a preparation to tackle the general case, we turn our attention to the case of two distinct active circles. We consider a \mit{} of the form
\[
Q=[(n,q),(\wh{n},\wh{q})],
\]
where $q$ and $\wh{q}$ have ramification orders 
$r$ and $\wh{r}$. We denote by  $q_0,\dots,q_{r-1}$ and $\wh{q}_0,\dots,\wh{q}_{\wh{r}-1}$ the elements of the Galois orbit of the corresponding \efs{}. A holomorphic family 
defined by $Q_b=[(n ,q(b)),(\wh{n}, \wh{q}(b))]$ is an admissible deformation if and only if $q(b)$ and $\wh{q}(b)$ are admissible deformations and if the rational numbers 
$\slope({q_i}-{\wh{q}_j})$ are constant.

We thus have to determine what are the slopes of the differences ${q_i}-{\wh{q}_j}$. This has been studied in \cite{doucot2021diagrams}, the results of which we now briefly recall. If $q$ and $\wh{q}$ are two distinct exponential factors, we can decompose them into a 
``common part'' and a ``different part'' 
in the following way. 
Recall that 
$\tau_k\colon\cI_d\to \cI_d$ denotes the the truncation map, 
discarding all monomials of slope $<k$, as in 
\eqref{eq: truncation}.
(Beware a different truncation was used in the proof of 
Prop. \ref{prop: one circle classn}, 
discarding monomials of slope $\le k$.)
Let 
$$E(q)\subset \IQ_{>0}$$
be the finite set of exponents occurring in $q$, so $q=\sum_{k\in E(q)} a_kx^k$ with each $a_k$ non-zero. 
Let $k\in E(q)$ be the smallest number such that
$$\<\tau_k(q)\>=\<\tau_k(\wh q)\>$$
i.e. the Galois orbits of the truncations are equal, if such a number exists.
If so then set
$$q_c = \tau_k(q),\qquad \wh q_c = \tau_k(\wh q).$$
If there is no such $k$ then set $q_c=\wh q_c=0$.
Then define $q_d=q-q_c$, $\wh q_d=\wh q -\wh q_c$, so that 
we get a decomposition 
\[ q=q_c+q_d, \qquad \wh{q}=\wh{q}_c+\wh{q}_d,
\]
of $q$ and $\wh{q}$ as the sum of a \textit{common part} $q_c$ and a \textit{different part} $q_d$. 
If $q_c=\wh{q}_c$ we say that $q$ and $\wh{q}$ are \textit{compatible}, as in Defn. \ref{defn: compatible pit}. 
Replacing $q$ or $\wh{q}$ by another element of their Galois orbit if necessary, we may assume without loss of generality that this is the case. 
Note that if $q$ and $\wh{q}$ do not have the same slope
then they do not have the same leading term up to Galois conjugacy, so $q_c=\wh{q}_c=0$.
If $q$ and $\wh{q}$ have a non-zero common part, in particular they have the same slope.  
We call the rational number 
\beq\label{eq: fission exponent}
f_{q,\wh{q}} := \max\bigl(\slope(q_d),\slope(\wh q_d)\bigr)
\in \IQ_{\ge 0}
\eeq
the \textit{fission exponent} of $q$ and $\wh{q}$.
It is zero if and only if $\<q\>=\<\wh q\>$.
Since it only depends on the circles/Galois orbits
this defines the fission exponent $f_{I,\wh I}$ for 
any Stokes circles 
$I=\<q\>$,$\wh I=\<\wh q\>$.

We are now in a position to describe the set of slopes we are interested in.

\begin{lem}
\label{lemma:slopes_differences}
Let $q$ and $\wh{q}$ be two exponential factors in distinct Galois orbits. The set of non-zero slopes among the rational numbers $\slope({q_i}-{\wh{q}_j})$, for $i=0,\dots, r-1$, $j=0,\dots, \wh{r}-1$ is equal to 
\[
\Levels(q_c) \sqcup \{f_{q,\wh{q}}\} \subset \Q_{>0},
\]
i.e it consists of the levels of the common part of $q$ and $\wh{q}$ together with their fission exponent. Furthermore, if $q$ and $\wh{q}$ are compatible i.e if $q_c=\wh{q}_c$ the map $(i,j)\mapsto \slope(q_i-\wh{q}_j)$ is entirely determined by the data of $\Levels(q_c)$ and $f_{q,\wh{q}}$.
\end{lem}

\begin{proof}
This follows directly from the proof of the Lemma 4.3 of \cite{doucot2021diagrams}. The main idea is that the levels of the common part are obtained from the differences between the different Galois conjugates of the common part, while the fission exponent is the slope of all the other differences (for which the Galois conjugates of the common part are the same). 
\end{proof}

As a consequence, the 
numerical equivalence relation on pointed irregular types can be clarified.

\begin{prop} 
\label{prop: deformation_two_active_circles}
Let $Q=[(n ,q),(\wh{n}, \wh{q})]$ be a \mit{} with two active circles, such that $q$ and $\wh{q}$ are compatible. Let $k :=  f_{q,\wh{q}}$ be the fission exponent of $q$ and $\wh{q}$. Then a 
pointed irregular type $Q'$ satisfies $Q'\sim Q$ if and only if it is of the form $Q'=[(n ,q'),(\wh{n}, \wh{q}')]$, with $q'$ and $\wh{q}'$ compatible and such that
\begin{enumerate}
\item $q\sim q'$ and $\wh{q}\sim \wh{q}'$, 
\item $q'_c=\wh{q}'_c$ satisfies $q'_c\sim q_c$,
\item $f_{q',\wh{q}'}=k$.
\end{enumerate}
\noindent
These three conditions hold if and only if: 
$\Levels(q)=\Levels(q')$, $\Levels(\wh q)=\Levels(\wh q')$, and $f_{q',\wh{q}'}=k$.
\end{prop}

\begin{proof} Let us assume that $Q'\sim Q$. Then, considering the differences internal to the two Galois orbits we have $q\sim q'$ and $\wh{q}\sim \wh{q}'$. Considering now the set of slopes of the differences between the two distinct Galois orbits for both $Q$ and $Q'$, we have from Lemma \ref{lemma:slopes_differences} that $L(q_c) \sqcup \{f_{q,\wh{q}}\}=L(q'_c) \sqcup \{f_{q',\wh{q}'}\}$ hence $f_{q',\wh{q}'}=f_{q,\wh{q}}=k$, and $L(q_c)=L(q'_c)$, so $q_c\sim q'_c$. Furthermore since $q$ and $\wh{q}$ are compatible, we have $\slope(q_0-\wh{q}_0)=k=\slope(q'_0-\wh{q}'_0)$ hence $q'$ and $\wh{q}'$ are also compatible. For the converse, let us assume that $Q'$ is of the claimed form. Then since $q\sim q'$ and $\wh{q}\sim \wh{q}'$ the slopes of the internal differences are the same for $Q$ and $Q'$. We have $L(q_c)=L(q'_c)$ and $f_{q,\wh{q}}=f_{q',\wh{q}'}$. Since $q$ and $q'$ are compatible, as well as $\wh{q}$ and $\wh{q}'$, the second part of Lemma \ref{lemma:slopes_differences} implies that the slopes of the differences between the distinct Galois orbits are the same for $Q$ and $Q'$.  
\end{proof}

From the case of a single circle, we know how to make the first two conditions explicit. Let us now investigate the third condition in more detail. 
Choose $k\in \Q_{>0}$ and let $q_c$ be an exponential factor whose exponents are all strictly greater than $k$. 
Write $k=n/m$ with $n,m$ coprime integers. 
Choose  $a,\wh{a}\in \IC$  and consider the exponential factors 
$$q :=  q_c + a z^k + b,\qquad \wh{q} :=  q_c +  \wh{a} z^k + \wh b$$
where $b,\wh b$ are exponential factors of slope $< k$.
The conditions for the fission exponent 
$f_{q,\wh q}$ to equal $k$ are as follows.

\begin{prop}\label{prop: condns at fissioexpt}
\label{proposition:types_of_fission_two_factors}
Let $r=\ram(q_c)$, $k=n/m$.
Then $f_{q,\wh q}=k$ if and only if either 1) or 2) holds:
\begin{enumerate}
\item $m$ divides $r$ and $a\neq \wh{a}$, or 

\item $m$ does not divide $r$ and either:
\begin{enumerate}
\item Exactly one of $a, \wh{a}$ is zero, or  
\item Both $a\neq 0$ and $\wh{a}\neq 0$, and 
furthermore  
$a^N\neq \wh{a}^N$, 
where $N=\lcm(r,m)/r$. %

\end{enumerate} 
\end{enumerate}
\end{prop}

\begin{remark}
Notice that in case 1., one of $a$ and $\wh{a}$ can be equal to zero. 
Also note that 2a,2b can be combined into the single statement that $a^N\neq \wh{a}^N$.
The three cases are distinguished since in 1
the number $k$ is in neither of the sets 
$\Levels(q),\Levels(\wh q)$, for 2a it is in just one, and for 2b it is in both.
For later use (cf. Prop \ref{prop: 3types of branching})
we encode the three cases 1,2a,2b pictorially as follows:
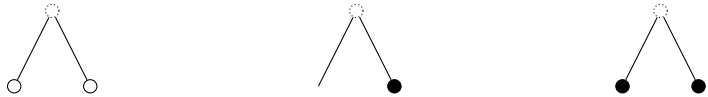
\begin{center}
\begin{tikzpicture}
\tikzstyle{mandatory}=[circle,fill=black,minimum size=5pt,draw, inner sep=0pt]
\tikzstyle{authorised}=[circle,fill=white,minimum size=5pt,draw,inner sep=0pt]
\tikzstyle{empty}=[circle,fill=black,minimum size=0pt,inner sep=0pt]
\tikzstyle{indeterminate}=[circle,densely dotted,fill=white,minimum size=5pt,draw, inner sep=0pt]
    \node[authorised] (A1) at (1-4,0){};
    \node[authorised] (A2) at (2-4,0){};
    \node[indeterminate] (A0) at (1.5-4,1){};   
    \node[empty] (B1) at (1,0){};
    \node[mandatory] (B2) at (2,0){};
    \node[indeterminate] (B0) at (1.5,1){};   
    \node[mandatory] (C1) at (1+4,0){};
    \node[mandatory] (C2) at (2+4,0){};
    \node[indeterminate] (C0) at (1.5+4,1){};   
  \foreach \from/\to in {A1/A0,A2/A0,B1/B0,B2/B0,C1/C0,C2/C0}
    \draw (\from) -- (\to);
\end{tikzpicture}
\end{center}

\end{remark}

\begin{proof} 
Clearly we need $a\neq \wh a$, and can set $b=\wh b=0$ without loss of generality.
We then need to see when $\<q\>\neq \<\wh q\>$. 
\begin{enumerate}
\item Let us first assume that $m$ divides $r$, so that $k$ is not a level of $\<q\>$ nor $\<\wh q\>$, and  
$\ram(q)=\ram(\wh{q})=\ram(q_c)$.
Thus the Galois orbits of $q$ and $\wh q$ are in bijection with that of $q_c$ (via truncation). 
This implies that the Galois orbits of $q$ and $\wh{q}$ are distinct if and only if $a\neq \wh{a}$.
\item 
2a) is clear so we consider 2b).  
Let us set $r'=\mathrm{lcm}(r,m)$ and $N=r'/r$.
$\Ram(q)$ is equal to $r'$ if $a\neq 0$, otherwise it is equal to $r$, and similarly for $\wh{q}$.  
If $a\neq 0$, then in the Galois orbit of $q$  
there are $N$ elements giving rise, upon truncation,  to any given element of the Galois orbit of $q_c$, and their coefficients of exponent $k=n/m$ differ by an $N$-th root of unity. The conclusion follows.
\end{enumerate}
\end{proof}

This enables us to get an explicit description of 
$\pureconfig(Q)$, as we now do for a few examples. 
In terms of lists of exponents attached to $I=\<q\>$,$\wh I=\<\wh q\>$
we have the subsets 
$$
L(I)\subset A(I)\subset \IR_{>0},\qquad
L(\wh I)\subset A(\wh I)\subset \IR_{>0}
$$
(as in \eqref{eq: L in A}) 
which we should identify just above the fission exponent
$f_{I,\wh I}$, in order to get the set of exponents whose coefficients 
we can vary.
And these coefficients 
can be varied arbitrarily provided those from $L(I)$ or $L(\wh I)$ remain non-zero and those at the fission exponent continue to satisfy the same 
part of Prop. \ref{prop: condns at fissioexpt}.

\begin{eg}
Let us look at a few examples illustrating the different cases: 
\begin{itemize}
\item Consider $Q=[(1,\lambda x^{s/r}),(1,\mu x^{s/r})]$ with $r>1$ and $s$ coprime and $\lambda\neq \mu e^{2in\pi/r}$ for any integer $n$ (so that the Galois orbits are disjoint). The common part of the two exponential factors is empty, and their fission exponent is the level $s/r$, so this fits into the case \textit{2(b)}. The (slope bounded) admissible deformations of $Q$ are of the form $Q'=[(1,q'_1),(1,q'_2)]$
with
\[
q'_1=\sum_{k=1}^s a_k x^{k/r}, \qquad q'_2=\sum_{k=1}^s b_k x^{k/r},
\]
with $a_1,\dots, a_s, b_1,\dots, b_s\in \IC$, $a_s\neq 0$, $b_s\neq 0$ and $a_s\neq b_s e^{2in\pi/r}$ for any integer $n$, that is we have 
\[\pureconfig(Q)\cong \{ (a_1,\dots, a_s,b_1,\dots,b_s) \in \C^{2s}\; \vert\; a_s\neq 0, b_s\neq 0, \forall n\in \N, a_s\neq b_s e^{2in\pi/r}\}.
\]
\item Let us consider $Q=[(1,q_1),(1,q_2)]$ with $q_1=x^{3/2}+x^{1/2}$, and $q_2=x^{3/2}+2x^{1/2}$. The common part of $q_1$ and $q_2$ is $x^{3/2}$, and they are compatible.  This fits into the first case, indeed the fission exponent $1/2$ is not a level of $q_1$ and $q_2$. The 
 (slope bounded) admissible deformations of $Q$ are of the form $Q'=[(1,q'_1),(1,q'_2)]$, with 
\[q'_1=a x^{3/2}+b x + c x^{1/2}, \qquad q'_1=a x^{3/2}+b x + d x^{1/2},
\]
with $a,b,c,d\in \IC$, $a\neq 0$ and $c\neq d$, i.e. we have 
\[
\pureconfig(Q)\cong \{ (a,b,c,d)\in \C^4\; \vert\; a\neq 0, c\neq d\}.\]
\item Let us consider $Q=[(1,q_1),(1,q_2)]$ with $q_1=x^{3/2}+x^{1/3}$, and $q_2=x^{3/2}$. The ramification order of $\cir{q_1}$ is 6, the common part of $q_1$ and $q_2$ is $x^{3/2}$, and they are compatible. This fits into case \textit{2(a)}, indeed the fission exponent $1/3$ is a level of $q_1$, but does not appear in $q_2$. The  (slope bounded) admissible deformations of $Q$ are of the form $Q'=[(1,q'_1),(1,q'_2)]$, with 
\[q'_1=a x^{3/2}+b x + c x^{1/2} + d x^{1/3}+ e x^{1/6}, \qquad q'_2=a x^{3/2}+b x + c x^{1/2},
\]
with $a,b,c,d,e\in \IC$, with $a\neq 0$ and $d\neq 0$, i.e. we have
\[\pureconfig(Q)\cong \{ (a,b,c,d,e)\in \C^5\; \vert\; a\neq 0, d\neq 0\}.\]
\end{itemize}
\end{eg}

\subsection{Fission data}

As a step towards the general case
we will give a first attempt at packaging the relevant data. 
Recall from Defn. \ref{defn: level datum} 
that a ``level datum'' is a finite, possible empty, subset 
$L\subset \IQ_{>0}$ satisfying the conditions 
\eqref{eq: condition on levels of a circle} 
(so the numbers it contains are the possible levels of a single Stokes circle).

\begin{defn}
A ``fission datum'' is a pair $\cF=(\cL,f)$ 
where $\cL$ is a multiset\footnote{Recall that a multiset is a set with multiplicities. 
Here this means the $L_i$ are not necessarily 
distinct level data 
(but the ordering of the level data, i.e. the labelling by indices $1,\ldots,m$, is not part of the data).}
$$\cL=L_1+\cdots + L_m$$ of level data
and $f$,  the fission exponents, are 
the choice of a rational number 
$f_{ij}=f_{ji}\in \IQ_{\ge 0}$,  for all
 $i,j\in \{1,\ldots,m\}$.
\end{defn}

An irregular class determines a fission datum in the obvious way, as follows.
If $\Th=\sum_1^m I_i$ is a rank $n$ irregular class (where the Stokes circles $I_i$ are not necessarily distinct) then define 
$L_i=L(I_i)$ to be the level datum of $I_i$ for each $i$, 
and define
$\cL(\Th)  :=  \sum_1^m L_i$
to be the corresponding multiset of level data. 
Then by taking $f_{ij}=f_{I_i,I_j}$ to be corresponding fission 
exponents \eqref{eq: fission exponent}, this determines the fission datum  
$\cF(\Th)=(\cL(\Th),f)$
of the irregular class $\Th$.
Note that the multiplicity of any 
given Stokes circle $I_j$ in the class 
$\Th$ is determined by the fission data by the recipe:
$$\Th(I_j) = \abs{\{i=1,2,\ldots,m\st f_{ij}=0\}}.$$

In turn  a {\em labelled} fission datum $\wh\cF=(\wh \cL,f)$ is a pair consisting of an ordered list
$$\wh \cL = [(n_1,L_1),\ldots,(n_p,L_p)]$$
where the $L_i$ are not necessarily distinct level data, 
together with fission exponents
$f_{ij}=f_{ji}\in \IQ_{\ge 0}$ 
for $i,j\in \{1,\ldots,p\}$
such that $f_{ij}=0$ if and only if $i=j$.
Then a \mit{} determines a labelled fission datum in the obvious way (and in turn a labelled fission datum determines a fission datum by forgetting the labelling).

The study in the case of one and two circles then implies 
one of the main statements:

\begin{thm}
\label{thm:admissible_deformation_iff_same_fissiondata}
Let $Q=[(n_1,q_1),\dots,(n_p,q_p)]$ be a \mit{}, which we assume to be compatible, and $\wh \cF$ its labelled fission datum. 
Then a \mit{} $Q'$ satisfies $Q'\sim Q$ if and only if it
is compatible and its 
labelled fission datum equals $\wh\cF$. 
\end{thm}
\begin{proof}
This follows from our study of the case of one and two active circles. To see this, notice that the fission data of $Q$ is equivalent to the data of the levels of its active circles, together with the data of the common part and the fission exponent of any pair of distinct active circles. 
The result now follows from Corollary \ref{cor: deformation_exp_factor_iff_same_levels} and Proposition \ref{prop: deformation_two_active_circles}.
\end{proof}

To go further we will now define fission trees
(gluing just above the fission exponents, as above); 
this will
give a way to parameterise the set of coefficients we can vary, and thus to describe the configuration spaces,  
leading to a proof that two irregular classes are admissible deformations of each other if and only if their fission data are equal.
It will also 
give a way to classify the possible topological data, i.e. the set of possible admissible deformation classes (it seems difficult to write down axioms for the fission data that actually arise from irregular classes, without discussing trees).

\subsection{Fission trees in the twisted setting}
\label{ssn: fission trees}

First we will describe abstractly the exact types of trees we get, 
and then define how to obtain such trees from irregular types/classes.

\subsubsection{Fission trees} 
Consider a six-tuple $(\cT,\IV,\IA,\IL, h, n)$ where:

$\bullet$\ $\cT$ is a metrised tree\footnote{See e.g. \cite{baker-faber2004} for metrised graphs, but note that ours are not compact, 
and recall that a tree is a special type of graph.},  
with vertices $\IV\subset \cT$,

$\bullet$\ $\IA\subset \IV$ is a subset (the admissible vertices),

$\bullet$\ $\IL\subset \IA$ is a finite, possibly empty, subset
(the internal levels/mandatory vertices),

$\bullet$\ $h\colon\cT\to \IR_{\ge 0}$ is a length preserving map, the height map, mapping each edge isomorphically onto an interval, such that $\IV_0 :=  h^{-1}(0)\subset \cT$ is a finite set and is the set of leaves of $\cT$,

$\bullet$\ $n$ is a map $\IV_0\to \IN_{>0}$, giving a multiplicity to each leaf.

\

The edges $E=E(\cT)=\pi_0(\cT\setminus \IV)$ of $\cT$
are the components of the complement of the vertices.
Thus any vertex that is not a leaf is adjacent to $\ge 2$ edges, one of which is the ``parent'' edge, and the others are the descendant edges.
The branch vertices $\IY\subset \IV$ are those with $\ge 2$ descendants
(where the branching of the tree occurs).  
The trunk of the tree is the union of all the edges and vertices 
above all the branch vertices. 
The vertices in $\IL$ will be called ``mandatory'', 
those in $\II := \IA\setminus \IL$ will be called ``\authorised{}'', 
and the others
$(\IV\setminus \IA)$  will be called ``empty''.

The ``full branch'' $\cB_i$ of any leaf $i\in \IV_0$ is 
the (minimal) subset of $\cT$ all the way from $i$ to the far end of the trunk. 
Let $\IL_i=\IL\cap \cB_i$ denote 
the internal levels on the $i$th full branch, 
and let $\IA_i=\IA\cap \cB_i$ similarly
(the admissible vertices on the $i$th full branch).

\ 

\begin{defn}\label{defn: fissiontree}
Such a tuple $(\cT,\IV,\IA,\IL, h, n)$ is a ``fission tree'' if:

1) $\IV = h^{-1}(\{0\}\cup h(\IA))$; the vertices are 
exactly the leaves plus the points that map to $h(\IA)$, 

2) $h$ maps each full branch isomorphically onto $\IR_{\ge 0}$,

3) The internal levels of any full branch map to a 
set of levels, i.e. $L_i :=  h(\IL_i)\subset \IQ_{>0}$ 
satisfies the conditions 
\eqref{eq: condition on levels of a circle},
for any leaf $i$ 
(so they are the possible levels of a single Stokes circle),

4) $A_i :=  h(\IA_i)\subset \IQ_{>0}$ is the set $A(L_i)$ of admissible exponents of $L_i$ for each leaf $i$, as in 
\eqref{eq: adm exps},

5) The children $\Ch(v)\subset \IV$ of each branch vertex $v\in \IY$ satisfy one of the following three conditions:

\ \ \ \ 1. All the vertices in $\Ch(v)$ are \authorised{},

\ \ \ \ 2a. One vertex in $\Ch(v)$ is empty and the others are mandatory,

\ \ \ \ 2b. All the vertices in $\Ch(v)$ are mandatory.
\end{defn}

Note in particular that the leaves of a fission tree are empty. 
Two fission trees are {\em isomorphic} if there is an 
isomorphism between the underlying trees relating all the data 
$\IV,\IA,\IL, h, n$.

A {\em labelling} 
of a fission tree with nodes $\IV_0$ is a total ordering of the 
set of leaves, i.e. a bijection 
$\psi\colon\{1,\ldots,\abs{\IV_0}\}\cong \IV_0$. 
Two labelled fission trees are isomorphic if there is a label-preserving isomorphism (so there is at most one isomorphism between labelled fission trees).

\begin{remark}\label{rmk/dfn: relram}
Note that the definition 
\eqref{eq: ramification indices on branch} 
extends directly to define the 
ramification index $\Ram(v)$ of any point 
$v\in\cT$ of a fission tree,
taking the $\lcm$ of the denominators of 
the heights of its mandatory ancestors $\ge v$ (i.e. 
in $\IL$, of height $\ge h(v)$ and on the same full branch).
If $p\in\IV$ is the parent of some vertex $v\in\IV$, by definition the ``relative ramification'' of $p$ is 
$\Ram(v)/\Ram(p)$.
Observe that the integer $N=\lcm(r,m)/r$ in
part 2a of Prop. \ref{prop: condns at fissioexpt} is an example of relative ramification.
\end{remark}

\begin{remark}
Axioms 3,4) of Defn. \ref{defn: fissiontree}
imply that axiom 5) could be replaced by the simpler statement ``$\Ch(v)$  contains at most one empty vertex for  any $v\in\IY$''. 
\end{remark}

\subsubsection{Fission data of a fission tree}
Given a fission tree $\cT=(\cT,\IV,\IA,\IL,h,n)$ 
let $A=h(\IA)\subset \IQ_{>0}$ 
be the admissible exponents. 
Given two distinct leaves $i,j\in \IV_0$, let 
$v_{ij}\in \IY$ be their nearest common ancestor 
(i.e. the branchpoint where $\cB_i,\cB_j$ meet).
Thus $\cT$ determines a number,  the fission exponent
$$f_{ij} :=  \precr(h(v_{ij}))\in A$$ 
for each pair of leaves, where $\precr\colon A\to A\cup \{0\}$
takes $a\in A$ to the preceding element of $A$, 
i.e. the largest element $<a$ (or to zero if $a=\min(A)$).
Axiom 5) implies  $f_{ij}\neq 0$.

Thus a fission tree $\cT$ determines a fission datum $\cF(\cT)=(\cL,f)$
where $\cL=\sum n_i L_i$ is the set of level data of each full branch (repeated according to their multiplicities) 
and $f$ encodes the fission exponents between the branches of the tree.
The following statement is now an exercise:
\begin{lem}\label{lem: fission trees and data}
Two fission trees are isomorphic if and only if their fission data are equal: 
$$\cT_1\cong \cT_2\  \Longleftrightarrow \ \cF(\cT_1) =  \cF(\cT_2).$$
\end{lem}
Similarly two labelled fission trees are isomorphic if and only if
their labelled fission data are equal.

\subsubsection{Fission tree of an irregular class} 
We now describe how to define the fission tree of an irregular class 
$\Th=\sum n_i I_i$: 
firstly there is a full branch $\cB_i$ 
(of multiplicity $n_i$) for each distinct circle $I_i$.  
Thus $\cB_i$ is a copy of $\IR_{\ge 0}$ equipped with the subsets 
$\IL_i\subset \IA_i\subset \cB_i$, and an isomorphism 
$h\colon\cB_i\to \IR_{\ge 0}$, defined so that 
the subsets $\IL_i\subset \IA_i$ map onto 
the sets $L_i=L(I_i)\subset A_i=A(I_i)$ of levels and admissible exponents
(defined from $I_i$ as in \ref{eq: L in A}).

We then define 
$A=A(\Th) := \bigcup A_i\subset \Q_{>0}$ to be the union of  all the admissible exponents. 
This is a discrete subset and for any $k\in \IQ_{>0}$ we can define 
the successor $\succr(k)\in A$ to be the next element of $A$, i.e. the smallest element of $A$ that is $>k$.
If $f_{ij}=f_{I_i,I_j}$ is the fission exponent between 
${I_i,I_j}$ then define the {\em gluing exponent}
$$g_{ij}  :=  \succr(f_{ij})\in A$$
to be the next admissible exponent after the fission exponent.

We then glue the full branches 
$\cB_i,\cB_j$ over the interval $[g_{ij}, \infty)$ for each $i,j$, to define the tree $\cT$ equipped with the map $h\colon\cT\to \IR_{\ge 0}$.
The subsets $\IL_i\subset \IA_i$ fit together to define 
$\IL\subset \IA\subset \cT$, and we set 
$\IV=h^{-1}(A\cup \{0\})\subset \cT$.
Let $\IV_k=h^{-1}(k)$ denote the vertices of height $k$.

This defines the fission tree 
$\cT(\Th)=(\cT,\IV,\IA,\IL,h,n)$. 
All the axioms are clear except 5), 
which will follow from Prop. \ref{prop: 3types of branching} below.

In case we start with 
a \mit{} $Q=[(n_1,q_1),\ldots,(n_m,q_m)]$, and not just an irregular class,  then we get a labelled fission tree $\wh \cT(Q)$, 
by labelling the nodes $\IV_0$ according to the labelling of the exponential factors $q_1,\ldots,q_m$.

The nodes $\IV\subset \cT$ may be interpreted 
in terms of truncated circles as follows. 
Recall that if $k\in \IQ$ then  $\tau_k(q)$
is the truncation, forgetting monomials of slope $<k$. 
\begin{lem}
For each $k\in A\cup \{0\}$
\beq
\IV_k\ \cong\  \{\cir{\tau_k(q_1)},\dots, \cir{\tau_k(q_m)}\}
\eeq  
i.e. the vertices of height $k$ are in bijection with the set of 
Galois orbits of 
the exponential factors truncated at $k$. 
\end{lem}
\pf
Consider two distinct 
circles $\<q_1\>$,$\<q_2\>$, 
corresponding to two leaves.
Consider the minimal element $k$ of $A$ such that
$\<\tau_k(q_1)\>=\<\tau_k(q_2)\>$.
Then, by definition, $k=g_{12}=\succr(f_{12})$.
\epf

In particular if $k>l$ are two  admissible exponents (or zero), 
we have a surjective map $\phi_{kl} \colon\IV_l\onto \IV_k$ defined by 
$\phi_{kl}(\cir{\tau_l(q_i)})= \cir{\tau_k(q_i)}$, and this determines 
the structure of the tree, by defining the unique parent of each node (if $k,l\in A$ are consecutive).
This also proves all the gluings of the full branches 
can be done consistently.

From this viewpoint, two elements 
$\cir{\tau_l({q_i})}$, $\cir{\tau_l({q_j})}\in \IV_l$ 
are descendants of the same vertex in $\IV_k$, where $l<k$ 
if they have the same truncation to exponent $k$, 
i.e. if $\cir{\tau_k({q_i})}=\cir{\tau_k({q_j})}$. 
Furthermore, if $\cir{q_i}$ and $\cir{q_j}$ are two active circles, corresponding to two leaves, in $\IV_0$, their closest common ancestor in the tree corresponds to (the Galois orbit of) their common part (this follows immediately from the definition of the common part of two exponential factors).

\begin{remark}\label{rmk: eggers}
Note that the visual image of a tree is clear by thinking about the eigenvalues of a matrix of meromorphic functions
($\sim$ a meromorphic Higgs field).
On the differential equations side
this idea is embedded in the ``fission'' picture, 
thinking about the growth/decay of the functions
$\exp(q(x))$ as $x\to \infty$ along a ray, and 
can be traced back to the Stokes diagram 
in \cite{stokes1857} p.116, 
reproduced on the cover of \cite{twcv}
(see also the pictures in \cite{fission, tops}).
However our exact definition of fission tree 
is quite subtle, in order for the main 
results to follow cleanly
(i.e. Thm. \ref{thm: charn of realzns}, 
\eqref{eq: explicit config} parameterising the 
configuration spaces in terms of points of the fission tree
and in turn Cor. \ref{cor: config prod decomp}, 
giving the product decomposition).
In particular the simpler definition of trees
(as in \cite{wall-eggers}) 
on the singularity theory 
(spectral curve) side of the wild nonabelian Hodge correspondence
are much less useful for either of these 
aims. 
Specifically if we took the
definition in \cite{wall-eggers} 
and transposed it to our setting (shifting and truncating suitably), then the resulting definition is too local: the location of the branchpoints $\IY$ would then just depend on the adjacent full branches so the product decomposition will not work cleanly and moreover 
some of the points of the tree that should parametrise 
distinct coefficients are identified.%
\end{remark}
\subsection{Truncated fission trees}\label{ssn: truncatedtrees}

Since any integer is an admissible exponent of any exponential factor, the set of admissible exponents of any irregular type $Q$ 
is unbounded from above.
For the configuration spaces we are interested in, the admissible exponents
are bounded by the Poincar\'e--Katz rank.  
Thus we will consider the ``truncated fission tree'' 
$\cT^\flat=\cT^\flat(Q)$ 
by defining the {\em root vertex} to be that of height
\beq\eta := \floor{\Katz(Q)+1}\eeq
and then removing all nodes/edges above the root
(and marking the root as empty/in-admissible). 
Note that $\eta$ 
is the smallest integer greater than
the Poincar\'e--Katz rank of $Q$ (the largest slope),
so the admissible nodes $\IA^\flat\subset\cT^\flat$, the subset of $\IA$ below the root,
are exactly the nodes that will contribute to the configuration space 
(as these are the admissible nodes of height 
$\le \Katz(Q)$). 
When drawing pictures of trees we will truncate as above, but also we 
will stop at the smallest admissible exponent, so the leaves will not be drawn. 
If we are just given a fission tree (and not an irregular type/class)
then we will use 
the ``minimal truncation'' at $\eta=\floor{k+1}$ where 
$k=\Katz(\cT) := \max(\{0, h(\IL), \max(h(\IY))-1\})$.
Indeed one can check that in the trace-free case 
$\Katz(\cT)$ is the Poincar\'e--Katz rank of any irregular class with fission tree $\cT$.
Using this we can attach an integer to any fission tree: 
\begin{defn}\label{defn: moduli number}
The {\em moduli number} $\mu(\cT)$ of a fission tree is
one plus the number of admissible vertices below the root, 
minus the number of integers below the root height: 
\begin{align}
\mu(\cT)  := &   1+ \bigl\vert{\IA^\flat}\bigl\vert - (\eta-1) = 
\bigl\vert{\IA^\flat}\bigr\vert+2 -\eta\\
=& 1+ \bigl\vert A_1\setminus \IN \bigl\vert + 
\sum_2^m \bigl\vert A_i\cap [0,f_i] \bigl\vert
\end{align}
\end{defn}
\noindent
It is clear that $\mu$ just depends on the fission tree and 
not the irregular class. The dimension of 
the special configuration space is $\mu-1$. 
In practice the second formulation is more convenient 
(when the tree is labelled by $1,\ldots,m$)
where $A_i=h(\IA_i)$, $f_i = \min_{j<i}(f_{ij})$
and $\bigl\vert A_1\setminus \IN \bigl\vert$ is as in 
\eqref{eq: non-int adm formula}. 

\begin{eg} \label{eg: bigexample}
Consider the   pointed irregular type 
$Q=[(1,q_1),\dots,(1,q_4)]$ with 
\[
q_1=x^{3/2}+x,\quad q_2=x^{3/2}+2x, \quad q_3=x^{1/3}, \quad q_4=2x^{1/3}
\] 
having rank $10$. Thus the root height $\eta=2$ and the set of admissible exponents of $Q$ smaller 
than $\eta$ is $\{3/2,1,1/2,1/3\}$. 
The labelled fission tree is drawn in Fig. 
\ref{ex:first_example_fission_tree}. We will always draw the mandatory vertices (the internal levels $\IL$)  as black %
circles, the \authorised{} vertices ($\IA\setminus\IL$) 
as white 
circles, and the empty vertices without any decoration. The heights are indicated on the left. The root is drawn as a black square.
\begin{figure}[h]
\begin{center}
\begin{tikzpicture}
\tikzstyle{mandatory}=[circle,fill=black,minimum size=5pt,draw, inner sep=0pt]
\tikzstyle{authorised}=[circle,fill=white,minimum size=5pt,draw,inner sep=0pt]
\tikzstyle{empty}=[circle,fill=black,minimum size=0pt,inner sep=0pt]
\tikzstyle{indeterminate}=[circle,densely dotted,fill=white,minimum size=5pt,draw, inner sep=0pt]
\tikzstyle{root}=[fill=black,minimum size=5pt,draw,inner sep=0pt]
    \node[root] (R) at (1.5,4){};
    \node[mandatory] (A1) at (0.5,3){};
    \node[empty] (A2) at (2.5,3){};
    \node[authorised](B1) at (0,2){};
    \node[authorised] (B2) at (1,2){};
    \node[authorised] (B3) at (2.5,2){};
    \node[authorised] (C1) at (0,1){};
    \node[authorised] (C2) at (1,1){};  
    \node[empty] (C3) at (2.5,1){};
    \node[empty] (D1) at (0,0){};
    \node[empty] (D2) at (1,0){};  
    \node[mandatory] (D3) at (2,0){};
    \node[mandatory] (D4) at (3,0){};
  \foreach \from/\to in {R/A1,R/A2,A1/B1,A1/B2,A2/B3,B1/C1,B2/C2,B3/C3, C1/D1,C2/D2,C3/D3,C3/D4}
   \draw (\from) -- (\to);
    \draw (-1.3,3) node {$3/2$};
    \draw (-1.3,2) node {$1$};
    \draw (-1.3,1) node {$1/2$};
    \draw (-1.3,0) node {$1/3$};
    \draw (0,-1) node {$q_1$};
    \draw (1,-1) node {$q_2$};
    \draw (2,-1) node {$q_3$};
    \draw (3,-1) node {$q_4$};
\end{tikzpicture}
\caption{The fission tree $\cT^\flat$ associated to $Q$ 
(not drawn isometrically). The labelling corresponds to the numbering of the $q_i$. The multiplicities of the leaves are all equal to $1$.}\label{ex:first_example_fission_tree}
\end{center}
\end{figure}
\end{eg}
\noindent
We will see below that the configuration space has dimension $8$ (the number of nonempty nodes below the root), and 
further (in Thm. \ref{thm: charn of realzns}) that 
there are 
three types of conditions on these $8$ coefficients:
those from each black vertex $\bullet$ should be non-zero, 
those from the two inconsequential siblings $\circ$
should be distinct, and
those from the two mandatory siblings $\bullet$
should have distinct $N$th powers (where $N=3$ in this example). 
For the special configuration space, the dimension is 
$7=8-1$ since there is one positive integer height below the root.
The moduli number $\mu(\cT)$ is eight.

\subsection{Realisations of a fission tree}\

Suppose we are given a 
fission tree $\cT=(\cT,\IV,\IA,\IL,h,n)$, 
and a map $c\colon\IA\to \IC$
with finite support, i.e. $c(v)=0$ for all but a finite number 
of the admissible vertices $v\in \IA\subset \cT$.

Then for each leaf $i\in \IV_0$ of the tree
we can define an exponential factor
\beq\label{eq: exp factor from map c}
q_i = \sum_{v\in \IA_i} c(v)x^{h(v)}
\eeq
summing over the admissible vertices $\IA_i=\cB_i\cap \IA$
on the $i$th full branch.
Thus $c$ gives the coefficients of 
a list of exponential factors.
Thus if the tree is labelled 
by some isomorphism $\psi\colon\{1,\ldots,m\}\cong \IV_0$ then
we get an element 
\beq\label{eq: pit from map c}
Q_c = [(n_1,q_1),\ldots,(n_m,q_m)],
\eeq
where $q_i$ is determined by $c$ as above, and $n_i=n(i)$ 
is the multiplicity of the leaf $i$.

We will say that the coefficient map $c$ is a {\em realisation} 
of the tree $\cT$ if $Q_c$ is a \mit{} and 
$\cT(Q_c)\cong \cT$, i.e. if
the fission tree of $Q_c$ is isomorphic to $\cT$ 
(the labellings match up by construction).

Thus we wish to make explicit the conditions on $c$ for it to be a realisation.

In effect we just need to check that the tree determined by $Q_c$ has the desired branching and mandatory nodes. 
Let us focus on a single branchpoint.   
Let $l>k$ be two consecutive heights of the tree, and 
let $v\in \IV_l$ be a vertex with $n$ children 
$\Ch(v)=\{w_1,\dots, w_n\}\subset \IV_k$.

Let $q$ be the exponential factor determined by $c$ 
at the node $v$, so that $q=\tau_l(q_j)$ for any leaf $j$ that is a descendant of $v$. 
Let $q_i=q + c_iz^k$ (where $c_i=c(w_i)$) be the corresponding exponential factors of the children, 
$i=1,\ldots,n$. 
Thus for $c$ to be a realisation, the $c_i$ have to be such that the Galois orbits $\cir{q_i}$ are pairwise distinct. 
Proposition \ref{proposition:types_of_fission_two_factors} allows us to characterise when this is the case. Set $r^+ := \ram(q)$, and write $k=\frac{s^-}{r^-}$, with $s^-$ and $r^-$ coprime.

\begin{prop}\label{prop: 3types of branching}
\label{proposition:types_of_fission}
Let $q$ be an exponential factor with all its exponents greater than $k$, and let $c_1,\dots,c_n\in \IC$ and consider $q_i=q+c_i z^k$ 
for $i=1,\dots, n$. Now 

\begin{enumerate}
\item If $r^-$ divides $r^+$, then $k$ is not an internal level of any of the $q_i$. Then the $\cir{q_i}$ are distinct if and only if $c_i\neq c_j$ 
for $1\leq i<j\leq n$.
\item Otherwise, let us assume that $r^-$ does not divide $r^+$, and let $N := \frac{\mathrm{lcm}(r^-,r^+)}{r^+}$. Then $k$ is a level of $\cir{q_i}$ if and only if $c_i\neq 0$, otherwise if $c_i=0$, $\cir{q_i}\in \IV_k$ is an empty vertex. Now the $\cir{q_i}$ are distinct if and only if 
\begin{enumerate}
\item Either one of the $c_i$, $i=1,\dots, n$, say $c_{i_0}$, is equal to zero and we have $c_i\neq 0$ for $i\neq i_0$ and $c_i\neq  \zeta c_j$ 
for %
$i,j\in \{1,\ldots,n\}\setminus \{i_0\}$
for any $N$-th root of unity $\zeta$.
\item Or, all of the $c_i$ are non-zero, and we have $c_i\neq  \zeta c_j$ for $1\leq i<j\leq n$ for any $N$-th root of unity $\zeta$.
\end{enumerate}
\end{enumerate}
\end{prop}

\begin{proof}
This follows immediately from the corresponding cases in Proposition \ref{proposition:types_of_fission_two_factors}.
\end{proof}

This implies that there are only three possible types of 
fission at a vertex $v$ in a fission tree, 
which correspond to the three cases 
 1, 2(a), and 2(b) of the proposition, and they 
 yield the axiom 5) in the definition of fission tree. 
In the case 1, since $k$ is not an internal level 
of any of the $q_i$, all the corresponding vertices are \authorised{}, which corresponds to the picture below (the parent vertex is dotted on the picture to indicate that it could be mandatory, \authorised{} or empty).
 
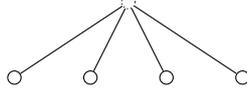
\begin{center}
\begin{tikzpicture}
\tikzstyle{mandatory}=[circle,fill=black,minimum size=5pt,draw, inner sep=0pt]
\tikzstyle{authorised}=[circle,fill=white,minimum size=5pt,draw,inner sep=0pt]
\tikzstyle{empty}=[circle,fill=black,minimum size=0pt,inner sep=0pt]
\tikzstyle{indeterminate}=[circle,densely dotted,fill=white,minimum size=5pt,draw, inner sep=0pt]
    \node[authorised] (A1) at (1,0){};
    \node[authorised] (A2) at (2,0){};
    \node[authorised] (A3) at (3,0){};
    \node[authorised] (A4) at (4,0){};
    \node[indeterminate] (B1) at (2.5,1){};   
  \foreach \from/\to in {A1/B1,A2/B1,A3/B1,A4/B1}
    \draw (\from) -- (\to);
\end{tikzpicture}
\end{center}

Otherwise, in the case $2$ the vertices corresponding to the $q_i$ are all mandatory provided they are non-empty. There are two possibilities: either, in the case 2(a), there is one empty vertex, which corresponds to the figure on the left below, or, in case 2(b) there is no empty vertex and all vertices are mandatory, which corresponds to the figure on the right below.

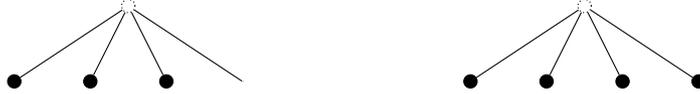
\begin{center}
\begin{tikzpicture}
\tikzstyle{mandatory}=[circle,fill=black,minimum size=5pt,draw, inner sep=0pt]
\tikzstyle{authorised}=[circle,fill=white,minimum size=5pt,draw,inner sep=0pt]
\tikzstyle{empty}=[circle,fill=black,minimum size=0pt,inner sep=0pt]
\tikzstyle{indeterminate}=[circle,densely dotted,fill=white,minimum size=5pt,draw, inner sep=0pt]
    \node[mandatory] (A1) at (1-3,0){};
    \node[mandatory] (A2) at (2-3,0){};
    \node[mandatory] (A3) at (3-3,0){};
    \node[empty] (A4) at (4-3,0){};
    \node[indeterminate] (B1) at (2.5-3,1){}; 
    \node[mandatory] (C1) at (1+3,0){};
    \node[mandatory] (C2) at (2+3,0){};
    \node[mandatory] (C3) at (3+3,0){};
    \node[mandatory] (C4) at (4+3,0){};
    \node[indeterminate] (D1) at (2.5+3,1){};   
  \foreach \from/\to in {A1/B1,A2/B1,A3/B1,A4/B1, C1/D1,C2/D1,C3/D1,C4/D1}
    \draw (\from) -- (\to);
\end{tikzpicture}
\end{center}

Thus we can write down the exact conditions 
for $c$ to be a realisation. 
Recall that two nodes are ``siblings'' if they have the same parent node.
In summary the result is the following:
 
\begin{thm}\label{thm: charn of realzns}
The map $c\colon\IA\to \IC$ is a realisation of $\cT$ if and only if 

1) $c(\IL)\subset \IC^*$, i.e.  $c(v)\neq 0$ for any mandatory node $v$,  

2) $c(u)\neq c(v)$ for any pair $u,v$ of \authorised{} siblings,

3) $c(u)^N\neq c(v)^N$ for any pair $u,v$ of mandatory siblings
where $N$ is the relative ramification of the parent of
$u,v$ (defined in Rmk. \ref{rmk/dfn: relram}).
\end{thm}
\pf
The first condition (and the definition of $\IA$) implies each full branch has the correct internal levels. Then 2) and 3) show that the tree of $Q_c$ has the right branching. 
\epf

Note that 2),3) can be combined into the single statement: 
if $u,v$ are admissible ($=$non-empty) siblings then 
$c(u)^N\neq c(v)^N$ where $N$ is the relative ramification of the parent of $u$ or $v$ 
(since the relative ramification of the parent of an \authorised{}
vertex is $1$).

\begin{cor}\label{cor: realisability}
Any fission tree admits a realisation.
\end{cor}
\pf
There are just a finite number of Zariski-closed conditions on the coefficients of any set of siblings, so the space of choices of $c$ is non-empty.
\epf

This immediately gives a clearer description of the configuration spaces.
First note that the definition of the fission tree implies:

\begin{lem}\label{lem: coeffrestrnsfromtree}
Suppose $Q$ is any compatible pointed irregular type
with  (labelled) fission tree  $\cT=\cT(Q)$, 
let $r=\Ram(Q)$, $x=t^r$ so that 
\beq\label{eq: naivecoeffs2}
Q := 
[
(n_1,q_1), \dots,
(n_m ,q_m)],\qquad q_i=\sum_{j=1}^s a_{i,j}t^{j},
\eeq 
for some collection of coefficients 
$\mathbf{a}=(a_{i,j})\in \IC$. 
Then there is a unique 
realisation $c=c_Q\colon\IA\to \IC$ of $\cT$
with $c(v)=a_{i,k}$  for all $i,k$ where 
$v=\<\tau_{k/r}(q_i)\>\in \IA$ is the vertex of $\cT$ 
determined
by the truncation of the exponential factor $q_i$.
\end{lem}
\pf
This amounts to verifying two conditions, which are now straightforward:
1) $a_{i,k}=0$ if $\<\tau_{k/r}(q_i)\>\in \IV\setminus\IA$, i.e. if the node of $\cT$ determined by the circle $\<\tau_{k/r}(q_i)\>$ is not admissible, and
2) $a_{i,k}=a_{j,k}$ if $\<\tau_{k/r}(q_i)\>=\<\tau_{k/r}(q_j)\>$, 
i.e. if the truncations determine the same node of $\cT$.
\epf

It follows that the configuration space
$\pureconfig(Q)$ 
of any compatible \mit{} $Q$ 
is isomorphic to the space
\begin{align}
\pureconfig(\cT^\flat)  := & \{c\colon\IA\to \IC\st \text{$c$ is a realisation of $\cT, \Katz(c)\le \Katz(Q) $}\} \label{eq: explicit config}
\\\notag
=& \{c\colon\IA^\flat\to \IC\st \text{1),2),3) of 
Thm. \ref{thm: charn of realzns} hold for $c$}\} 
\end{align}
of realisations of 
 the truncated fission tree $\cT^\flat$, 
Here $\IA^\flat$ is the set of admissible nodes of 
$\cT^\flat$ and 
$\Katz(c)=\max\{h(a)\st a\in \IA, c(a)\neq 0\}$
is the height of the realisation.
In other words we have established the following:
\begin{thm} \label{thm: configs from realzns}
If $Q$ is a compatible \mit{} 
then 
$\pureconfig(Q) \cong \pureconfig(\cT^\flat)$, where $\cT^\flat$ is the truncated 
labelled fission tree of $Q$.
\end{thm}
\pf 
By 
Thm. \ref{thm:admissible_deformation_iff_same_fissiondata} 
$\pureconfig(Q)$ is the set of
compatible \mit{} $Q'$ 
with the same labelled fission data as $Q$
and $\Katz(Q')\le \Katz(Q)$.
Then as in Lem. \ref{lem: fission trees and data} 
this is the same as saying 
the labelled fission tree of $Q'$ is isomorphic to that of $Q$ and $\Katz(Q')\le \Katz(Q)$.
Then by Lem. \ref{lem: coeffrestrnsfromtree}, any such $Q'$ arises uniquely as a realisation $c$ of $\cT^\flat$.
\epf

This immediately gives a product decomposition of the configuration space.
Given a compatible  \mit{} $Q$ 
with fission tree $\cT$, 
for any vertex $v\in \IV$ of $\cT$ 
let  $\Ch_\IA(v)=\IA\cap \Ch(v)$ be the set of 
admissible/nonempty children of $v$ and let
$\Ch_\bullet(v) = \IL\cap \Ch(v)$ be  the subset
of mandatory children of $v$ (black vertices). 
Define the local configuration space 
$\pureconfig_v(Q)=\pureconfig_v(\cT)$
for 
the vertex $v\in \IV$: 
\beq
\pureconfig_v(\cT)  :=  \left\{
c\colon\Ch_\IA(v)\to \IC\ \Bigl\vert\ 
\begin{matrix}
c(\Ch_\bullet(v))\subset \IC^*, \text{ and}\\ 
c(u)^N\neq c(w)^N \  \forall u\neq w\in \Ch_\IA(v)
\end{matrix}\right\}
\eeq
where $N$ is the relative ramification of $v$ 
(defined in Rmk. \ref{rmk/dfn: relram}).
$\pureconfig_v(\cT)$ is taken to be a point if $v$ has no non-empty children and otherwise it thus takes the form:
\beq
\pureconfig_v(\cT)\cong X_n := \{ a_1,\dots, a_n\in \IC \;\vert\; a_i\neq a_j \text{ for } i\neq j\},
\eeq
if $v$ has $n$ \authorised{} children, or
\beq
\pureconfig_v(\cT) \cong 
X^*_{n,N} := \{ a_1,\dots, a_n\in \IC \;\vert\; a_i\neq 0, \; a_i^N\neq a_j^N \text{ for } i\neq j\}
\eeq
if $v$ has $n$ mandatory children 
and relative ramification $N$.

\begin{cor}\label{cor: config prod decomp}
Let $Q$ be a compatible \mit{} with fission tree $\cT$ 
and let $\IV^\flat$ be the vertices of its truncated 
fission tree $\mathcal{T^\flat}$.  
The configuration space 
$\pureconfig(Q)$ admits a product decomposition:
\[
\pureconfig(Q)\cong\prod_{v\in \IV^\flat} \pureconfig_v(\cT).
\]
In particular the dimension of 
$\pureconfig(Q)$ is the number 
of admissible (nonempty) vertices of 
the fission tree $\cT$ of height $\le$ the 
Poincar\'e--Katz rank of $Q$.
\end{cor}
\pf
Since $\pureconfig(Q) \cong \pureconfig(\cT^\flat)$
this follows from the characterisation of the realisations 
$c\colon\IA^\flat\to \IC$ of $\cT^\flat$ given 
in Thm. \ref{thm: charn of realzns}.
\epf

In particular it follows that
the configuration spaces are {\em connected}, since each of the local configuration spaces 
$\pureconfig_v(\cT)\cong X_n$ or $X_{n,N}^*$ is connected.
This yields a combinatorial/topological 
characterisation of admissible deformations, 
as follows.

\begin{cor}\label{cor: trees and adm defmns of types }
Two compatible pointed irregular types are admissible deformations of each other if and only if they have isomorphic labelled fission trees, 
if and only if they have the same 
labelled fission data.
\end{cor}
\pf
Given $Q,Q'$ suppose that $\Katz(Q)\ge \Katz(Q')$
and consider $\pureconfig(Q)$. 
If the labelled fission trees are isomorphic then 
$Q'=Q_c$ for some realisation $c$ of $\cT(Q)$, 
by Theorem \ref{thm: configs from realzns}. 
Thus, by connectedness, 
$Q,Q'$ are admissible deformations of each other
with $\IB=\pureconfig(Q)$.
The converse is clear.
The last statement follows as in Lem. \ref{lem: fission trees and data}.
\epf

\begin{cor}\label{cor: adm defm types condns}
Two rank $n$ 
irregular classes are admissible deformations of each other if and only if they have isomorphic fission trees, 
if and only if they have the same fission data.
\end{cor}
\pf
This follows from Cor. \ref{cor: trees and adm defmns of types } by considering local lifts from irregular classes to pointed irregular types.
\epf

\begin{cor}\label{cor: config space as fine modspace }
Let $Q$ be a compatible \mit{}. Then the configuration space $\pureconfig(Q)$ is a fine moduli space of all \mits{} that are admissible deformations of $Q$, with Poincar\'e--Katz rank $\le \Katz(Q)$.
Similarly $\mathbf{SB}(Q)$ is a fine moduli space of all trace-free admissible deformations of any trace-free pointed irregular type $Q$.
\end{cor}
\pf
The main point is that the
product decomposition and the formulae 
\eqref{eq: exp factor from map c}, 
\eqref{eq: pit from map c}
give a universal family of pointed irregular types
over $\pureconfig(Q)=\pureconfig(\cT^\flat(Q))$.
\epf

\begin{eg}
Let us look at a few examples, starting with the ones with two exponential factors studied previously
\begin{itemize}
\item Consider $Q=[(1,\lambda x^{s/r}),(1,\mu x^{s/r})]$ with $r$ and $s$ coprime and $\lambda\neq \mu e^{2in\pi/r}$ for any integer $n$. The (labelled) fission tree $\mathcal{T}$ is the following:

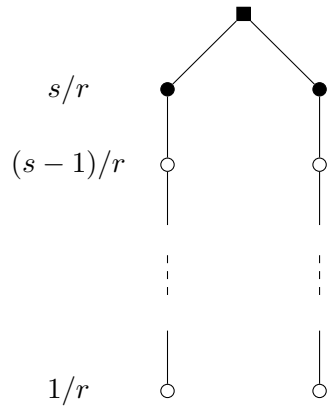
\begin{center}
\begin{tikzpicture}
\tikzstyle{mandatory}=[circle,fill=black,minimum size=5pt,draw, inner sep=0pt]
\tikzstyle{authorised}=[circle,fill=white,minimum size=5pt,draw,inner sep=0pt]
\tikzstyle{empty}=[circle,fill=black,minimum size=0pt,inner sep=0pt]
\tikzstyle{root}=[fill=black,minimum size=5pt,draw,inner sep=0pt]
\tikzstyle{indeterminate}=[circle,densely dotted,fill=white,minimum size=5pt,draw, inner sep=0pt]
    \node[root] (R) at (1,5){};
    \node[mandatory] (A1) at (0,4){};
    \node[mandatory] (A2) at (2,4){};
    \node[authorised] (B1) at (0,3){};
    \node[authorised] (B2) at (2,3){};
    \node[authorised] (C1) at (0,0){};
    \node[authorised] (C2) at (2,0){};  
  \foreach \from/\to in {R/A1,R/A2,A1/B1,A2/B2}
    \draw (\from) -- (\to);
    \draw (B1)--(0,2.2);
    \draw (B2) -- (2,2.2);
    \draw[dashed] (0,1.8) -- (0,1.2);
    \draw[dashed] (2,1.8) -- (2,1.2);
    \draw (0,0.8) -- (C1);
    \draw (2,0.8) -- (C2);
    \draw (-1.3,4) node {$s/r$};
    \draw (-1.3,3) node {$(s-1)/r$};
    \draw (-1.3,0) node {$1/r$};
\end{tikzpicture}
\end{center}

From the tree we read the space of admissible deformations: we have 
\[
\pureconfig(Q)\cong X^*_{2,r}\times \C^{2(s-1)},\qquad
X^*_{2,r} = \{a_1,a_2\in \IC\st a_1a_2\neq 0,\, a_1^r\neq a_2^r\}.
\]
Indeed, the factor corresponding to the root vertex is $X^*_{2,r}$, the factors for the leaves are trivial, while for all other vertices $v$ the space $\pureconfig_v(\cT)$ 
is isomorphic to $\IC$. 
Similarly $\mathbf{SB}(Q)\cong X^*_{2,r}\times 
\C^{N}$ where $N=2s-2-\floor{s/r}$, 
removing one dimension for each integer below the root.
\item Let us consider $Q=[(1,q_1),(1,q_2)]$ with 
$q_1=x^{3/2}+x^{1/2}$, and $q_2=x^{3/2}+2x^{1/2}$. The fission tree is drawn below.
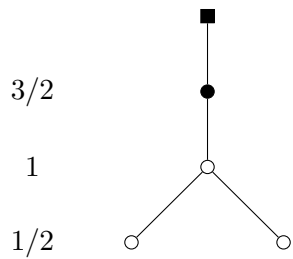
\begin{center}
\begin{tikzpicture}
\tikzstyle{mandatory}=[circle,fill=black,minimum size=5pt,draw, inner sep=0pt]
\tikzstyle{authorised}=[circle,fill=white,minimum size=5pt,draw,inner sep=0pt]
\tikzstyle{empty}=[circle,fill=black,minimum size=0pt,inner sep=0pt]
\tikzstyle{indeterminate}=[circle,densely dotted,fill=white,minimum size=5pt,draw, inner sep=0pt]
\tikzstyle{root}=[fill=black,minimum size=5pt,draw,inner sep=0pt]
    \node[root] (R) at (1,3){};
    \node[mandatory] (A1) at (1,2){};
    \node[authorised] (B1) at (1,1){};
    \node[authorised] (C1) at (0,0){};
    \node[authorised] (C2) at (2,0){};  
  \foreach \from/\to in {R/A1,A1/B1,B1/C1,B1/C2}
   \draw (\from) -- (\to);
    \draw (-1.3,2) node {$3/2$};
    \draw (-1.3,1) node {$1$};
    \draw (-1.3,0) node {$1/2$};
\end{tikzpicture}
\end{center}
From this we read that the space of admissible deformations satisfies
\[
\pureconfig(Q)\cong X^*_{1,2}\times X_1\times X_2\cong 
\C^*\times\C\times \{a,b\in \IC \;\vert\; a\neq b\}.
\]
\item Let us consider $Q=[(1,q_1),(1,q_2)]$ with 
$q_1=x^{3/2}+x^{1/3}$, and $q_2=x^{3/2}$. The fission tree is drawn below.
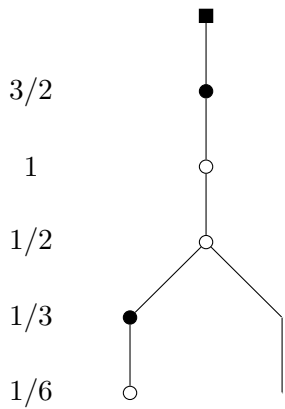
\begin{center}
\begin{tikzpicture}
\tikzstyle{mandatory}=[circle,fill=black,minimum size=5pt,draw, inner sep=0pt]
\tikzstyle{authorised}=[circle,fill=white,minimum size=5pt,draw,inner sep=0pt]
\tikzstyle{empty}=[circle,fill=black,minimum size=0pt,inner sep=0pt]
\tikzstyle{indeterminate}=[circle,densely dotted,fill=white,minimum size=5pt,draw, inner sep=0pt]
\tikzstyle{root}=[fill=black,minimum size=5pt,draw,inner sep=0pt]
    \node[root] (R) at (1,5){};
    \node[mandatory] (A1) at (1,4){};
    \node[authorised] (B1) at (1,3){};
    \node[authorised] (C1) at (1,2){};
    \node[mandatory] (D1) at (0,1){};  
    \node[empty] (D2) at (2,1){}; 
    \node[authorised] (E1) at (0,0){};  
    \node[empty] (E2) at (2,0){}; 
  \foreach \from/\to in {R/A1,A1/B1,B1/C1,C1/D1,C1/D2,D1/E1,D2/E2}
   \draw (\from) -- (\to);
   \draw (-1.3,4) node {$3/2$};
   \draw (-1.3,3) node {$1$};
    \draw (-1.3,2) node {$1/2$};
    \draw (-1.3,1) node {$1/3$};
    \draw (-1.3,0) node {$1/6$};
\end{tikzpicture}
\end{center}
\end{itemize}
This yields
\[
\pureconfig(Q)\cong X^*_{1,2}\times X^*_{1,3} \times X_1^3\cong (\C^*)^2\times \C^3.
\]

\end{eg}

\subsection{Topological skeleta}\label{ssn: skeleta}
Corollary \ref{cor: adm defm types condns} has the following immediate global consequence. 
Suppose $\mathbf{\Si}=(\Si,\ba,\mathbf{\Th})$ is a rank $n$ wild Riemann surface, with $\Si$ a compact Riemann surface, 
$\ba\subset \Si$ a finite subset, and 
$\mathbf{\Th}=\{\Th_a\st a\in \ba\}$ the data of a 
rank $n$ irregular class for each marked point. 
For each $a\in \ba$ let $\cT_a=\cT(\Th_a)$ be the fission tree of the irregular class $\Th_a$. 
Define the {\em topological skeleton} of 
$\mathbf{\Si}$ to be the pair
$$\text{Sk}(\mathbf{\Si}) = (g,\mathbf{F})$$
where $g\ge 0$ is the genus of $\Si$ and 
$\mathbf{F}=\sum_{a\in \ba} [\cT_a]$
is the {\em forest} of $\mathbf{\Si}$, i.e.
the multiset of isomorphism classes of 
fission trees determined by all the $\cT_a$,  
as $a$ ranges over the marked points $\ba\subset \Si$. 
In general it is a multiset rather than a set as some of the fission trees at distinct points may be isomorphic. 
As explained above the notion of admissible deformations of (twisted) wild Riemann surfaces follows from the untwisted case of \cite{gbs} (extending 
the generic case in \cite{JMU81, Mal-imd2long}).
In brief a holomorphic admissible deformation
of rank $n$ wild Riemann surfaces 
over a base space $\IB$ 
is a holomorphic family 
$\pi:\underline{\Si}\to \IB$
of compact Riemann surfaces, together with a holomorphic 
multisection $\si\subset \underline{\Si}$
restricting to a finite subset 
$\ba_b\subset \Si_b$ in each fibre 
$\Si_b=\pi^{-1}(b)\subset\underline{\Si}$, and also the choice of a rank $n$ irregular class
$\Th_a$ at each point $a\in \ba_b\subset \Si_b$
for each $b\in \IB$. 
These should be such that the irregular classes vary holomorphically, and the deformation is admissible: For the pointed surfaces the admissibility condition just means that each 
fibre $\Si_b$ is smooth, and none of the points $\ba_b$ coalesce (so we get the same number of points in each fibre). Finally we need to define what it means for the irregular classes to vary holomorphically and admissibly: these are local conditions so we can work over any small enough open subset $U$ of $\IB$ and focus on one marked point $a\in \ba_b$. We can then choose
a local coordinate $z$ vanishing at $a$ (for all $b\in U$). Thus we reduce to the situation in the definition of holomorphic admissible deformations of irregular classes 
given in Defn. \ref{defn: hol adm defmn} above.
In turn, two rank $n$ wild Riemann surfaces
will be said to be 
admissible deformations of each other
if they are related by the equivalence relation 
generated by the condition of being 
two fibres of a
holomorphic  admissible deformation (as defined above).

Corollary \ref{cor: adm defm types condns} then implies:

\begin{cor}
 Two rank $n$ wild Riemann surfaces are admissible deformations of each other if and only if they have the same topological skeleta.
\end{cor}
\pf
We may assume the genus and the number of marked points are equal, otherwise the result is clear.
Now if their topological skeleta are distinct then Cor. \ref{cor: adm defm types condns} 
implies there is no admissible deformation
between them.
Conversely if they have the same topological skeleta then we can use the universal family over  
Teichmuller space (or its version with marked points) to admissibly deform both 
wild Riemann surfaces so they have the same underlying Riemann surface with marked points, and moreover we may assume (since the topological skeleta are the same) that at each marked point
the two fission trees are isomorphic.
Finally we use the local statement 
(Cor. \ref{cor: adm defm types condns})
to deform the two irregular classes at each point in an admissible fashion, until they are equal.
\epf
In particular, since the set of possible topological skeleta is countable, this gives control over the set of possible topological types of the wild character varieties
$\mathcal{M}_\text{B}$: as in \cite{gbs} 
they form a local system of varieties over any admissible deformation, so up to isomorphism there is just one 
(Poisson) wild character variety for each possible topological skeleton (the key part of the proof in \cite{gbs} is local on the circle of directions so works equally well for twisted irregular classes).

\begin{remark}
The irregular Deligne--Simpson problem can then be 
stated as follows:
given a topological skeleton $(g,\mathbf{F})$, 
let $\mathbf{L}$ be the set of all the leaves of all the
trees in the forest $\mathbf{F}$. 
Choose a conjugacy class $\mathcal{C}_i\subset \GL_{n_i}(\IC)$ for each leaf $i\in \mathbf{L}$,
where $n_i\ge 1$ is the multiplicity of $i$.

{\bf Question:}\  for which choices of skeleton and conjugacy classes is there an irreducible algebraic connection $(V,\nabla)\to \Sigma^\circ=\Sigma\setminus \ba$
with the given topological skeleton and 
 formal monodromy conjugacy classes?
Passing to Stokes local systems \cite{tops}
(by the Stokes version of the 
irregular Riemann--Hilbert 
 correspondence),  
this can easily be rewritten as a linear algebra problem 
(as in \cite{gbs} \S9.4, 
and the graphical examples in \cite{cmqv} \S11).
\end{remark}

\section{Local wild mapping class groups}\label{sn: lpwmcgs}

\subsection{Pure local wild mapping class groups} 

Let $Q$ be a \mit{}. 
We define the \textit{pure} local (twisted) wild mapping class group of $Q$ as the fundamental group $\purelwmcg(Q) := \pi_1(\pureconfig(Q))$ of the configuration 
space of admissible deformations (with basepoint $\mathbf{a}_Q$). From the description of $\pureconfig(Q)$, it follows immediately that  $\purelwmcg(Q)$ also factorises as a product of factors associated to each vertex of the fission tree.
Note that the usual mapping class group may be defined in two ways, as the group of mapping classes or as the fundamental group of the moduli space/stack of Riemann surfaces; in the wild setting we only know the second type of definition, by generalising the notion of Riemann surface by incorporating non-trivial irregular classes.

\begin{thm}\label{thm: wmcg product}
Let $\mathcal{T}$ 
be the fission tree of $Q$
and let $\IV^\flat$ be the nodes of its truncation. 
We have
\[
\purelwmcg(Q)\cong \prod_{v\in \IV^\flat} \purelwmcg_v(\mathcal{T}),
\]
with $\purelwmcg_v(\mathcal{T}) := \pi_1(\pureconfig_v(\mathcal{T}))$.
\end{thm}

Since $\pureconfig_v(\mathcal{T})$ is 
isomorphic to a hyperplane complement of 
the form $X_n$ or $X^*_{n,N}$, setting
\[
\purelwmcg_n := \pi_1(X_n), \qquad \purelwmcg^*_{n,N} := \pi_1(X^*_{n,N}),
\]
we get that $\purelwmcg(Q)$ is always a product of factors of the form $\purelwmcg_n$ or $\purelwmcg^*_{n,N}$. Notice that $\purelwmcg_n$ is none other but the pure braid group on $n$ strands $PB_n$. We recover in this way the untwisted case of \cite{doucot2022local}, since it corresponds to the case where all vertices are \authorised{}, hence only the factors $\purelwmcg_n$ will appear in the factorisation. In comparison with the untwisted case, new factors $\purelwmcg^*_{n,N}$ appear as factors of the local wild mapping class group.

Let us now look at a few examples. A simple case to look at is when there is only one exponential factor.

\begin{prop} Let $q$ be an exponential factor. 
Then $\purelwmcg(q) := \pi_1(\pureconfig(q))$  
is isomorphic to $\mathbb{Z}^{\abs{\Levels(q)}}$.
\end{prop}

\begin{proof}
This follows immediately from Prop. \ref{prop: singcirc conf}. 
\end{proof}

\begin{eg}
Let us look once again at our previous examples:
\begin{itemize}
\item Consider $Q=[(1,\lambda x^{s/r}),(1,\mu x^{s/r})]$ with $r$ and $s$ coprime and $\lambda\neq \mu e^{2in\pi/r}$ for any integer $n$.
The space of realisations is homotopy equivalent to 
\[
X^*_{2,r}=\{a,b\in \C^* \;\vert\; \forall k\in \Z, a\neq b e^{2ik\pi/r}\}
\]
The fundamental group is thus $\purelwmcg(Q)\cong\purelwmcg^*_{2,r}$.

\item Let us consider $Q=[(1,q_1),(1,q_2)]$ with $q_1=x^{3/2}+x^{1/2}$, and $q_2=x^{3/2}+2x^{1/2}$. The space of realisations is homotopy equivalent to
\[
X^*_{1,2}\times X_2=\C^* \times \{(a,b)\in \C^2 \;\vert \; a\neq b\}.
\]
Its fundamental group thus satisfies $\purelwmcg(Q)\cong\purelwmcg^*_{1,2}\times \purelwmcg_{2}\cong \Z^2$.
\item Let us consider $Q=[(1,q_1),(1,q_2)]$ with $q_1=x^{3/2}+x^{1/3}$, and $q_2=x^{3/2}$. The space of admissible deformations is homotopy equivalent to $X^*_{1,2}\times X^*_{1,3}\cong (\C^*)^2$, so its fundamental group is isomorphic to $\Z^2$.
\end{itemize}
\end{eg}

\begin{remark}\label{rmk: bmr etc}
It turns out that the new building blocks $\purelwmcg^*_{n,N}$ which appear in the twisted case coincide with some braid groups studied in the literature on complex reflections, in particular by Broué--Malle--Rouquier 
\cite{BMR1998complex}.
More precisely, the group $\purelwmcg^*_{n,N}$ is the same as the group denoted $P(N,1,n)$ there, and the hyperplane complement $X^*_{n,N}$ is equal to the one denoted by $\mathcal{M}^{\#}(N,n)$ there (introduced in their Lemma 3.3). 
Our study of the local wild mapping class groups thus gives a modular interpretation (in 2d gauge theory) for this class of 
complex braid groups 
(coming from hyperplane arrangements of the reflecting hyperplanes of these complex/unitary reflection groups). 
We will see below that the corresponding 
complex reflection groups, 
the generalised symmetric groups 
$S(n,N)=G(N,1,n)$, appear also in our setting, 
when passing from irregular types to irregular classes.
\end{remark}

\subsection{\Nonpure{} local wild mapping class groups}

Given an irregular class $\Th$ we have defined a 
fission tree $\cT=\cT(\Th)$
and this determines
a configuration space
$\pureconfig(Q)\cong \pureconfig(\cT)\subset 
\Map(\IA^\flat,\IC)$
where $\IA^\flat$ is the finite 
set of admissible nodes of the truncated fission tree.
Now we will define a finite group $W(\cT)$ 
(the {\em Weyl group} of the tree) 
and a free action of $W(\cT)$ on $\pureconfig(\cT)$ 
so that two points $Q_1,Q_2\in \pureconfig(\cT)$ 
are in the same orbit if and only if $[Q_1]=[Q_2]$, i.e. if 
they determine the same irregular class.
This leads to the {\em full} local wild mapping class group. 

For certain simple examples of fission trees $\cT$ we will then find
$$W(\cT)\cong \Sym_n\ltimes (\IZ/N\IZ)^n$$
i.e. the Weyl group is a so-called generalised symmetric group 
$S(N,n)$, isomorphic to 
the complex reflection group denoted $G(N,1,n)$  in the Shephard--Todd classification \cite{Shep-Todd} (they are the symmetry groups of the 
regular complex polytopes
called the generalised cubes  $\gamma_n^N$ and the 
generalised octahedra $\beta_n^N$, see 
e.g. \S13.4, p.147 of Coxeter's book 
\cite{coxeter-regcxpolytopes} on regular complex polytopes).

\subsubsection{The Weyl group of a fission tree}

Let $\cT$ be a fission tree, let  $p=\abs{\IV_0}$ be the number of leaves of 
$\cT$, 
and choose a labelling $\psi\colon\{1,\ldots,p\}\cong \IV_0$.
Let $v_i=\psi(i)$ be the $i$th leaf, and let 
$v_{ij}\in \cT$ be the branchpoint where the full branches 
$\cB_i,\cB_j$ meet, i.e. the nearest common ancestor of $v_i,v_j$. 
Let $r_i=\Ram(v_i)\in \IN$ be the ramification index of the $i$th leaf, and let $r_{ij}=\Ram(v_{ij})$, for all $i,j=1,2,\ldots,p$, so that 
$r_{ij}$ divides both $r_i$ and $r_j$.

The group $\Aut(\cT)$  of automorphisms of $\cT$ 
embeds in the symmetric group 
$\Sym_p=\Aut(\IV_0)$ since any automorphism of the tree 
is determined by its action on the leaves. 
Thus $\Aut(\cT)$ acts on 
the product $(\IZ/r_1\IZ)\times\cdots\times(\IZ/r_p\IZ)$ 
of cyclic groups, permuting the factors 
(since if two full branches are isomorphic then they have the same ramification index $r_i$).
Thus we can consider the
 semi-direct product
\beq\label{eq: sdp for W}
\Aut(\cT)\ltimes \bigl((\IZ/r_1\IZ)\times\cdots\times  (\IZ/r_p\IZ)\bigr)
\eeq
defined via this action.
The Weyl group of $\cT$ is the following subgroup of this semi-direct product.

\begin{defn}
The Weyl group of the fission tree $\mathcal{T}$ is the subgroup of 
\eqref{eq: sdp for W} 
defined by
\[
W(\mathcal{T}) := \left\{ (\pi,(d_1,\dots,d_p))\in \Aut(\mathcal{T})\ltimes (\Z/r_1\Z\times \cdots\times \Z/r_p\Z) \st
d_i \equiv d_j \text{ mod } r_{ij}\right\}.
\]
\end{defn}

Note that since $r_{ij}\!\st\! r_i$ there is a 
quotient map 
$\pr_i\colon\IZ/r_i\IZ\onto \IZ/r_{ij}\IZ$ and the statement that 
$d_i \equiv d_j \text{ mod } r_{ij}$ just means that 
$\pr_i(d_i)=\pr_j(d_j)$.
If $p=1$ then $W(\cT)=\IZ/r_1\IZ$.

In the rest of this section we will prove the following.

\begin{thm}\label{thm: full lwmcg}
Let $Q$ be a compatible pointed irregular type with fission tree $\cT$.
The Weyl group $W(\cT)$ acts freely on the configuration space 
$\pureconfig(Q)=\pureconfig(\cT)$ and the quotient 
$$\fullconfig(\Th)=\fullconfig(\cT) := \pureconfig(Q)/W$$
is the space of all irregular classes that are admissible deformations of 
$\Th := [Q]$ with bounded Poincar\'e--Katz rank.
\end{thm}

It follows that $\fullconfig(\Th)$ is a manifold and we can define the {\em full} local wild mapping class group to be 
\beq
\fulllwmcg(\Th) = \pi_1(\fullconfig(\Th)).
\eeq

\subsubsection{Action on pointed irregular types}

The semi-direct product \eqref{eq: sdp for W} 
is easy to understand via its action on pointed irregular types. 
Let $$Q=[(n_1,q_1),\ldots,(n_p,q_p)]$$ 
be a pointed irregular type with fission 
tree $\cT$ as above.
Let $G$ denote the corresponding group \eqref{eq: sdp for W}
defined as a semi-direct product.
If $g=(\pi,\bd)\in G$ with $\bd=(d_1,\ldots,d_p)$
then we can obtain another  
pointed irregular type $g\cdot Q$ with fission tree $\cT$ by the formula:
\begin{align*}
g\cdot Q &= (\pi,0)\cdot 
[(n_1,\si^{d_1}(q_1)),\ldots,(n_p,\si^{d_p}(q_p))]\\
&=[
(n_{\pi(1)},\si^{d_{\pi(1)}}(q_{\pi(1)})),\ldots,
(n_{\pi(p)},\si^{d_{\pi(p)}}(q_{\pi(p)}))
]
\end{align*}
so that the cyclic groups rotate the choices of ``pointing'' and $\pi$ permutes the exponential factors which have isomorphic full branches.

Note that  $g\cdot Q$ will always have the same fission tree as $Q$ but it may not be an admissible deformation of $Q$, i.e. it may not be a point of the configuration space 
$\pureconfig(Q)$.
The Weyl group $W(\cT)$ is the subgroup characterised by this property:

\begin{lem}\label{lem: charn of weyl group}
Suppose $Q$ is a compatible pointed irregular type, 
and $g\in G$ is an element of the semi-direct product
\eqref{eq: sdp for W}.
Then $g\cdot Q$ is an admissible deformation of $Q$ 
(i.e. $g\cdot Q\in\pureconfig(Q)$) 
if and only if $g\in W(\cT)$. 
\end{lem}
\pf
This amounts to characterising the $g\in G$ such that
$g\cdot Q$ is still compatible since 1)
Any admissible deformation of $Q$ will still be compatible, and 2) by Lemma \ref{lem: coeffrestrnsfromtree}, any compatible $g\cdot Q$ will be in 
$\pureconfig(Q)$. 
Now to see if $g\cdot Q$ is still compatible, 
we need the exponential factors to ``branch'' like the circles they determine (i.e. their Galois orbits), 
and this comes down to requiring  
$$\tau_k(\si^{d_i}(q_i))=\tau_k(\si^{d_j}(q_j))$$
where $k=g_{ij}$ is the height of the nearest common
ancestor of the leaves $i,j$ in $\cT$, 
for all indices $i\neq j$.
But this just says that 
$\si^{d_i}(q_c) = \si^{d_j}(q_c)$ where
$q_c=\tau_k(q_i)=\tau_k(q_j)$ is the common part of 
$q_i,q_j$.
Now since $r_{ij}$ is the ramification order of $q_c$
this just means that $d_i\equiv d_j$ modulo $r_{ij}$, as in the definition of $W(\cT)$.
\epf

This implies that the finite group $W(\cT)$ acts on 
the configuration space $\pureconfig(Q)\cong \pureconfig(\cT)$, and we can now prove the rest of the theorem.

\ 

\pfms (of Thm \ref{thm: full lwmcg}).
It remains to show the action is free, and the 
orbits in $\pureconfig(Q)$ 
are the subsets with the same irregular class. 
The action is free since 1) the circles corresponding to isomorphic full branches are indeed distinct circles (as else they would be recorded in the multiplicity of the
leaf of the full branch), so the permutation is trivial, and 2) the $r_i$ are indeed the exact sizes of the Galois orbits of the $q_i$, so no smaller cyclic shift will act trivially.
Finally note that the pointed irregular types with given irregular class are just related by a choice of ordering of the circles, and the pointings of each circle, 
are related by the group $G$. 
Thus Lemma \ref{lem: charn of weyl group}
implies the $W$ orbits in 
$\pureconfig(Q)\cong \pureconfig(\cT)$
are exactly the points with the same irregular class.
\epfms

\begin{remark}
The Weyl group $W(\cT)$ is thus a
subgroup of the symmetric group of all 
permutations of the exponential factors in the corresponding full irregular type (the leaves of the corresponding full/naive fission tree, as in the proof of Prop. \ref{prop: one circle classn}, closely related to the 
``$3d$ fission tree'' in \cite{tops}).
\end{remark}

\begin{eg}\label{eq: getgensymgp}
Suppose $Q=[(1,q_1),\ldots,(1,q_p)]$ 
with $q_i=a_ix^{s/N}$, $i=1,\ldots,p$ 
where $(s,N)=1, N>1$ and the $a_i$ are generic complex numbers (in the sense that $a_i\neq 0$ and  $a_i^N\neq a_j^N$ if $i\neq j$).
Then the top part of  $\cT^\flat(Q)$ looks as follows, with $p$ mandatory nodes branching from the root: 

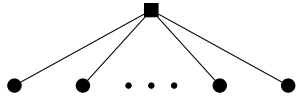
\begin{center}
\begin{tikzpicture}
\tikzstyle{dotted}= [dash pattern=on \pgflinewidth off 5pt]
\tikzstyle{mandatory}=[circle,fill=black,minimum size=5pt,draw, inner sep=0pt]
\tikzstyle{authorised}=[circle,fill=white,minimum size=5pt,draw,inner sep=0pt]
\tikzstyle{empty}=[circle,fill=black,minimum size=0pt,inner sep=0pt]
\tikzstyle{root}=[fill=black,minimum size=5pt,draw,inner sep=0pt]
\tikzstyle{ellps}=[circle,fill=black,minimum size=2pt,draw,inner sep=0pt]
\tikzstyle{indeterminate}=[circle,densely dotted,fill=white,minimum size=5pt,draw, inner sep=0pt]
    \node[root] (R) at (1,5){};
    \node[mandatory] (A1) at (-0.8,4){};
    \node[mandatory] (A2) at (0.1,4){};
    \node[mandatory] (A3) at (1.9,4){};
    \node[mandatory] (A4) at (2.8,4){}; 
\node[ellps] (A5) at (0.7,4){}; 
\node[ellps] (A6) at (1,4){}; 
\node[ellps] (A7) at (1.3,4){}; 
  \foreach \from/\to in {R/A1,R/A2,R/A3,R/A4}
    \draw (\from) -- (\to);
\end{tikzpicture}
\end{center}
Thus $\Aut(\cT)=\Sym_p$, and  in turn
$W(\cT)=\Sym_p\ltimes (\IZ/N\IZ)^p$ is the generalised symmetric group $S(N,p)\cong G(N,1,p)$, since the 
ramification indices $r_{ij}$ are all equal to $1$. 
For example any {\em symmetric irregular class}  
$I(a\col b)  :=  \sum_{i=1}^m\<\eps^i x^{a/b}\>$ falls into this setting 
($p=m, N=b/m$).
Here $a,b$  are positive integers with highest common factor $m$, 
and $\eps=\exp(2\pi i/b)$.
These are the classes obtained by pulling back 
a Stokes circle of the form $\<w^{1/b}\>$, 
under the cyclic covering $w=x^a$, and occur for the Molins--Turrittin differential equation $y^{(n)}=x^{\nu}y$, which has the same exponential factors as the irregular class $I(n+\nu \col n)$ (up to an overall scale factor of $n/(n+\nu)$) \cite{molins1876, turrittin1950}. 
\end{eg}

As in the untwisted case \cite{doucotrembado2023topology}, there is an explicit recursive description of the automorphism groups of the fission trees, and in turn the Weyl group.  
Define a ``maximal subtree'' of a fission tree $\cT$ to be one of the trees obtained by removing the highest branch node of 
$\cT$ (and all the higher edges), so the root of the subtree was the highest branch vertex of $\cT$.
The function $\Ram$ on $\cT$ (from Rmk. \ref{rmk/dfn: relram}) restricts to define a function on any such subtree, so its Weyl group is well-defined.

%


\begin{thm} Let $\mathcal{T}$ be a fission tree.

If $\cT$ consists only of one full branch whose leaf has ramification order $r$ then 
$\mathrm{Aut}(\mathcal{T})$ is trivial and 
$W(\cT)$ is isomorphic to $\Z/r\Z$.

Otherwise, let 
$\wb{\mathcal{T}}_1,\dots, \wb{\mathcal{T}}_s$ be the distinct isomorphism classes of decorated trees among its maximal subtrees, 
and for $i=1,\dots, s$ let $n_i \in \N$ denote the number of such maximal subtrees having the isomorphism class 
$\wb{\mathcal{T}}_i$.   
Then $\Aut(\cT)$ is a product of wreath products:
\[
\mathrm{Aut}(\mathcal{T})\cong\prod_{i=1}^s \Sym_{n_i}\wr \mathrm{Aut}(\wb{\mathcal{T}}_i).
\]
\item
In turn if $r$ is the ramification order of the root
of all the subtrees $\wb{\mathcal{T}}_i$ then
\[
W(\mathcal{T})\cong\left\{(\pi_i,(g_{i,1},\dots, g_{i,n_i}))\in\prod_{i=1}^s \Sym_{n_i}\wr 
W(\wb{\mathcal{T}}_i)
\, \Bigl\vert\, \begin{matrix}\delta(g_{i,k})\equiv \delta(g_{j,l}) \;\mathrm{mod}\; r \\  \forall i\neq j, \forall k,l\; \end{matrix} \right\},
\]
where $\delta(g_{i,k})$ denotes the shift at the root of the subtree induced by the automorphism $g_{i,k}$.
\end{thm}

\begin{proof} For the untwisted automorphism group the proof is exactly the same as the one in \cite{doucotrembado2023topology} for the twisted case. 
For the Weyl group, the proof is similar, the only difference is that the compatibility conditions for the root shifts $\delta(g_{i,k})$ imply that we must restrict to a subgroup of the wreath product 
$\prod_{i=1}^s \Sym_{n_i}\wr W(\wb{\mathcal{T}}_i)$ that we would get otherwise.
\end{proof}

The general picture we finally arrive at is the following: the full  local wild mapping class group $\fulllwmcg(\Theta)$ is an extension of the Weyl group $W(\cT)$ of the fission tree by the pure local wild mapping class group  $\purelwmcg(\Theta)$, i.e. we have a short exact sequence 
\beq
1\longrightarrow \purelwmcg(Q) \longrightarrow \fulllwmcg(\Th)\longrightarrow W(\cT) \longrightarrow 1
\eeq

\subsubsection{The case of one active circle}

In general, the exact sequence does not split, so it is not easy to get a fully explicit description of the \nonpure{} 
local wild mapping class group.
In the simple case of only one exponential factor however, we will now see it is possible to be more explicit and to determine completely 
the \nonpure{} local mapping class group:

\begin{thm}  Let $\Th=\<q\>$ be an irregular class with one circle  determined by the 
exponential factor $q$. 
Then the  full local wild mapping class group 
$\fulllwmcg(\Th)$  is isomorphic to $\mathbb{Z}^{|\Levels(q)|}$.
\end{thm}

\begin{proof}
Let $r=\ram(q)$ and $q_i=\si^i(q), i=0,\dots, r-1$ denote the Galois orbit of $q=\sum a_ix^{i/r}$. 
Let us write $L(q)=\Levels(q)=(k_1>\cdots>k_m)$ and $k_i=n_i/r$. 
Let us consider the loop 
$\gamma_i\colon[0,1]\to \pureconfig$, $i=1,\dots,m$, 
such that $\gamma_i$ makes the coefficient $a_{n_i}$ of $x^{k_i}$ 
go once around the origin, and leaves the other coefficient constant. 
Let us also consider the path $\nu$ in $\pureconfig$ 
such that 
\[
\nu(t)=
\sum_j a_j e^{-2\sqrt{-1}\pi jt/r} x^{j/r},
\]
so that $\nu(0)=q, \nu(1)=q_1$. 
Then $\fulllwmcg(q)$ is the abelian group generated by the homotopy classes determined by $\ga_1,\dots \ga_m$ and $\nu$, while the subgroup 
$\purelwmcg(q)$ is generated by $\ga_1,\dots \ga_m$. 
There are no relations between the generators  $\ga_1,\dots \ga_m$, 
which recovers the fact that 
$\purelwmcg(q)\cong \mathbb{Z}^{m}=\mathbb{Z}^{\abs{L(q)}}$.  
On the other hand the family  $\ga_1,\dots \ga_m,\nu$ is not free. Indeed, if we follow $r$ times the loop determined by $\nu$, 
it is the image of a loop upstairs, going from 
$q_0$ to itself, with the coefficient $a_{n_i}$ going around the origin a number of times equal to $\text{gcd}(n_i,r) =: d_i$. We thus have the following relation between the $m+1$ generators
\[
d_1 \ga_1 + \dots + d_m \ga_m= r \nu
\]
(we have used an additive notation here since the group is abelian). Using the fact that the greatest common divisor 
$\text{gcd}(r,n_1,\dots, n_m)=1$, the Schmidt algorithm used to classify finitely generated abelian groups transforms the vector 
$(d_1,\dots,d_m,-r)\in \mathbb{Z}^{m+1}$ 
corresponding to this relation into $(1,0,\dots,0)$, which implies 
that $\fulllwmcg(q)\cong \mathbb{Z}^{m}=\mathbb{Z}^{|L(q)|}$.
\end{proof}

In particular, the short exact sequence here reads
\[
0 \longrightarrow \Z^{|L(q)|}\longrightarrow \Z^{|L(q)|}\to \Z/r\Z\longrightarrow 0,
\]
and does not split.

\section{Outlook}

Several of the directions we plan to pursue are as follows:

1) Extend this work beyond type $A$ to any $G$: the notion of irregular class is already in \cite{twcv}, the analogue of fission trees for any $G$ in the untwisted case is in 
\cite{doucot2022local,doucotrembado2023topology}
and the definition of admissible deformations 
will again follow from 
that in the untwisted case \cite{gbs}. 
Presumably this will lead to other examples 
of (non-real) complex reflection braid groups.

2) Apply the fission trees to 
the Lax project \cite{hit70}, classifying the (wild) 
nonabelian Hodge spaces up to isomorphism/deformation.
For example how many distinct deformations classes 
are there in each complex dimension $2,4,6,\ldots$? 
Can the fission trees be ``combined'' with the diagrams of 
\cite{diags, doucot2021diagrams} (which are invariant under Fourier--Laplace) to give a refined invariant?
This encompasses the question of classifying isomonodromy systems, and the Painlev\'e equations are amongst the dimension $2$ examples.

3) Study further the full moduli spaces/stacks 
$\mathfrak{M}_{g,\mathbf{F}}$ 
of admissible deformations of any wild Riemann surface, whose fundamental groups are the wild mapping class groups (as in \cite{gbs}, \cite{p12} \S8),
generalising the Riemann moduli spaces 
$\mathfrak{M}_{g,m}$ and $\mathfrak{M}_{g,\{m\}}$ 
in the tame case (where $\{m\}$ means $m$ unordered marked points), as well as their universal covers (analogues of Teichm\"uller spaces). 
For example for $g=0$ and $ \mathbf{F}=[\cT]$ a single tree, 
this just amounts to quotienting the configuration space
$\fullconfig(\cT)$ by the two dimensional group 
of M\"obius transformation fixing one point of the Riemann sphere.
All the  Painlev\'e equations (in their standard Lax representations)
are especially nice since their (trace-free) 
moduli spaces $\mathfrak{M}_{g,\mathbf{F}}$
have dimension one, so are {\em wild modular curves}, reflecting the fact they are ODEs not PDEs
(their time variable 
ranges over a finite cover of this moduli space). 

4) Finally 
we are interested in quantising the symplectic/Poisson local systems of wild character varieties 
$\uMB \to \mathbb{B}\to \mathfrak{M}_{g,\mathbf{F}}$ (and the corresponding de Rham isomonodromy connections) 
to get linear representations of the wild mapping class groups
(see \cite{BabKit, bafi, jns08, yamakawa-qslims, rembado-slqc, rembado-syms, felder-rembado} for some examples).

\def\cprime{$'$} \def\cprime{$'$} \def\cprime{$'$} \def\cprime{$'$}
\providecommand{\bysame}{\leavevmode\hbox to3em{\hrulefill}\thinspace}
\providecommand{\MR}{\relax\ifhmode\unskip\space\fi MR }
\providecommand{\MRhref}[2]{%
  \href{http://www.ams.org/mathscinet-getitem?mr=#1}{#2}
}
\providecommand{\href}[2]{#2}

\vspace{0.2cm} 
\noindent
P.~B.:\  Université Paris Cité and Sorbonne Université, CNRS, IMJ-PRG, Paris, France.\\ 
 \url{https://webusers.imj-prg.fr/~philip.boalch/} \hfill \email{boalch@imj-prg.fr}

\vspace{0.1cm} 
\noindent 
J.~D.:\ 
Group of Mathematical Physics, Faculty of Sciences, Universidade de Lisboa, Campo Grande, Edificio C6, PT-1749-016 Lisboa, Portugal. \\
\url{https://www.normalesup.org/~doucot/}
\hfill 
\email{jmdoucot@fc.ul.pt}

\vspace{0.1cm} 
\noindent
G.~R.:\ Institut Montpelliérain Alexander Grothendieck (IMAG), University of Montpellier, Place Eugène Bataillon 34090 Montpellier. \\
\url{https://sites.google.com/view/gabrielerembado}
\hfill
\email{gabriele.rembado@umontpellier.fr}

\end{document}